\documentclass[12pt]{amsart}

\usepackage{amssymb}
\usepackage{graphicx}
\usepackage{epstopdf}

\theoremstyle{plain}
\newtheorem{theorem}{Theorem}[section]

\theoremstyle{definition}
\newtheorem{definition}[theorem]{Definition}
\newtheorem{proposition}[theorem]{Proposition}
\newtheorem{remark}[theorem]{Remark}
\newtheorem{lemma}[theorem]{Lemma}
\newtheorem{corollary}[theorem]{Corollary}
\newtheorem{example}[theorem]{Example}
\newtheorem{problem}{Problem}
\newtheorem{conjecture}[problem]{Conjecture}

\newcommand{\M}{\mathcal{M}}

\newcommand{\G}{\mathcal{G}}

\newcommand{\q}{\mathbb{Q}}
\newcommand{\z}{\mathbb{Z}}
\newcommand{\Z}{\mathbb{Z}}
\renewcommand{\r}{\mathbb{R}}

\DeclareMathOperator{\LLS}{LLS}
\DeclareMathOperator{\lsin}{lsin}
\DeclareMathOperator{\il}{l\ell}
\DeclareMathOperator{\SL}{SL}
\DeclareMathOperator{\sign}{sign}
\DeclareMathOperator{\trace}{trace}

\title{Generalised Markov numbers}

\author{Oleg Karpenkov, Matty van-Son}
\date{04 September 2018}

\address{Oleg Karpenkov\\
University of Liverpool\\
Mathematical Sciences Building\\
Liverpool L69 7ZL, United Kingdom
} \email{karpenk@liv.ac.uk}

\address{Matty van-Son\\
	University of Liverpool\\
	Mathematical Sciences Building\\
	Liverpool L69 7ZL, United Kingdom
} \email{sgmvanso@liverpool.ac.uk}

\thanks
{
O.~Karpenkov is partially supported by EPSRC grant EP/N014499/1 (LCMH)}


\begin{document}
\input{epsf}
\begin{abstract}
In this paper we introduce generalised Markov numbers and
extend the classical Markov theory for the discrete Markov spectrum
to the case of generalised Markov numbers.
In particular we show recursive properties for these numbers
and find corresponding values in the Markov spectrum.
Further we construct a counterexample to the
generalised Markov uniqueness conjecture.
The proposed generalisation is based on geometry of numbers.
It substantively uses lattice trigonometry and geometric theory of
continued numbers.
\end{abstract}

\maketitle
\tableofcontents

\section{Introduction}

In this paper we develop a geometric approach to the classical theory on the discrete Markov spectrum in terms of the geometric theory of continued fractions.
We show that the principles hidden in the Markov's theory
are much broader and can be substantively extended
beyond the limits of Markov's theory.
The aim of this paper is to introduce the generalisation
of Markov theory and to make the first steps in its study.

\vspace{2mm}

{
\noindent
{\bf Markov minima and Markov spectrum.}
Let us start with some classical definitions.
Traditionally the Markov spectrum is defined via certain minima of binary quadratic forms (see~\cite{mar1,mar2}).
Let us define the Markov minima first.
}

\begin{definition}\label{Markov-minimum}
Let $f$ be a binary quadratic form with positive discriminant.
The \textit{Markov minimum of $f$} is
$$
m(f)=\inf \limits_{\z^2 \setminus \{(0,0) \}} |f|.
$$
\end{definition}

In some sense the Markov minimum gives us information on how far the locus of $f$
is from the integer lattice (except for the origin).
In this context it is reasonable to consider the following normalisation.

\begin{definition}
Let $f$ be a binary quadratic form with positive discriminant $\Delta(f)$.
The \textit{normalised Markov minimum of $f$} is
$$
\M(f)=\frac{m(f)}{\sqrt{\Delta(f)}}.
$$
The set of all possible values for
$$
\displaystyle\frac{1}{\M(f)}
$$ is called the \textit{Markov Spectrum}.
\end{definition}

The functional $\M$ has various interesting properties.
In particular if two forms $f_1$ and $f_2$ are proportional
then their normalised Markov minima are the same.
The lattice preserving linear transformations of the coordinates
also do not change the value of the normalised Markov minima.

By that reason, $\M(f)$ depends only on the integer type (see Definition~\ref{IntCong} below) of the arrangement of two lines
forming the locus of $f$, which has an immediate trace in geometry of numbers.

\vspace{2mm}

{
\noindent{\bf Some history and background.}
The smallest element in the Markov spectrum is $\sqrt{5}$.
It is defined by the form
$$
x^2+xy-y^2
$$
and therefore it is closely related to the {\it golden ratio} in geometry of numbers.
The first elements in the Markov spectrum in increasing order are
as follows:
$$
\sqrt{5}, \qquad \sqrt{8}, \qquad \frac{\sqrt{221}}{5},
\qquad \frac{\sqrt{1517}}{13},
\qquad \frac{\sqrt{7565}}{29}, \quad \ldots
$$
The first two elements of the Markov spectrum were found in~\cite{Korkine1873} by A.~Korkine, G.~Zolotareff.
It turns out that the Markov spectrum is discrete at the segment
$[\sqrt{5}, 3]$ except the element $3$.
This segment of Markov spectrum was studied by A.~Markov in~\cite{mar1,mar2}.
We discuss his main results below.
}

\vspace{2mm}

The spectrum above the so-called Freiman's constant
$$
F=
\frac{2221564096+283748\sqrt{462)}}{491993569}=4.527829\ldots
$$
contains all real numbers (see~\cite{Freiman1975}).
The segment $[3,F]$ has a rather chaotic Markov spectrum,
there are various open problems regarding it.
There are numerous open gaps in this segment (i.e. open segments
that do not contain any element of Markov spectrum).
For the study of gaps we refer to a nice overview~\cite{cus1}.

\vspace{2mm}

\begin{remark}
Note that most of the generalised Markov and almost Markov trees are entirely contained in the Markov spectrum above 3,
which evidences the fractal nature of the spectrum.
\end{remark}

The Markov spectrum has connections to different areas of mathematics,
let us briefly mention some related references.
Hyperbolic properties of Markov numbers were studied by C.~Series
in~\cite{Series1985}.
A.~Sorrentino, K.~Spalding, and A.P.~Veselov
have studied various properties of interesting monotone functions related to the Markov spectrum and the growth rate of values of binary quadratic forms
in~\cite{SV2017-2} and~\cite{SorV2017,SV2017}.
B.~Eren and A.M.~Uluda\u{g} have described some properties of Jimm for
certain Markov irrationals in~\cite{Buket2018}.
In his paper~\cite{Gayfulin2017}
D.~Gaifulin studied attainable numbers and the Lagrange spectrum
(which is closely related to Markov spectrum).

\vspace{2mm}

Finally let us say a few words about the multidimensional case.
One can consider a form of degree $d$ in $d$ variables
corresponding to the product of $d$ linear factors.
The Markov minima and the $d$-dimensional Markov spectrum here are defined as in the two-dimensional case.
It is believed that {\it the $d$-dimensional Markov spectrum
for $d>2$ is discrete}, however
this statement has not been proven yet.
We refer the interested reader to the original manuscripts~\cite{Davenport1938,Davenport1938a, Davenport1939,Davenport1941,Davenport1943} by H.~Davenport,
and~\cite{Swinnerton-Dyer1971} by  H.P.F.~Swinnerton-Dyer,
and to a nice overview in the book~\cite{Gruber1987}
by P.M.~Gruber and C.G.~Lekkerkerker.

\vspace{2mm}

{
\noindent{\bf Markov numbers and their properties.}
Let us recall an important and surprising theorem by A.~Markov \cite{mar1,mar2} which relates
the Markov Spectrum below 3 to certain binary quadratic forms and solutions to the
\textit{Markov Diophantine equation}
$$
	x^2+y^2+z^2=3xyz.
$$
}
\begin{definition}
The solutions of this equation are called {\it Markov triples}.
Elements of Markov triples are said to be {\it Markov numbers}.
\end{definition}

Markov triples have a remarkable structure of a tree. This is due to the following three facts:

\vspace{1mm}
{
\noindent
{\bf Fact 1.}
If $(a,b,c)$ is a solution to the Markov Diophantine equation then any permutation of $(a,b,c)$ is a solution as well.
}
\vspace{1mm}

{
\noindent
{\bf Fact 2.}
If $(a,b,c)$ is a solution to the Markov Diophantine equation then the triple $(a,3ab-c,b)$ is a solution as well.
}

\vspace{1mm}

{
\noindent
{\bf Fact 3.}
All possible compositions of the operations described in Facts 1 and 2 applied to the triple $(1,1,1)$
give rise to all positive solutions to Markov Diophantine equation.
}

\vspace{2mm}

Let us order the elements in triples $(a,b,c)$ as follows: $b\ge a\ge c$ and denote them by vertices.
We also connect the vertex $(a,b,c)$ by a directed edge to the vertices $(a,3ab-c,b)$ and $(b,3bc-a,c)$
Then we have an arrangement of all the solutions as a graph which is actually a binary directed tree with the ``long'' root,
see in Figure~\ref{fig-mark}.

\begin{figure}
\epsfbox{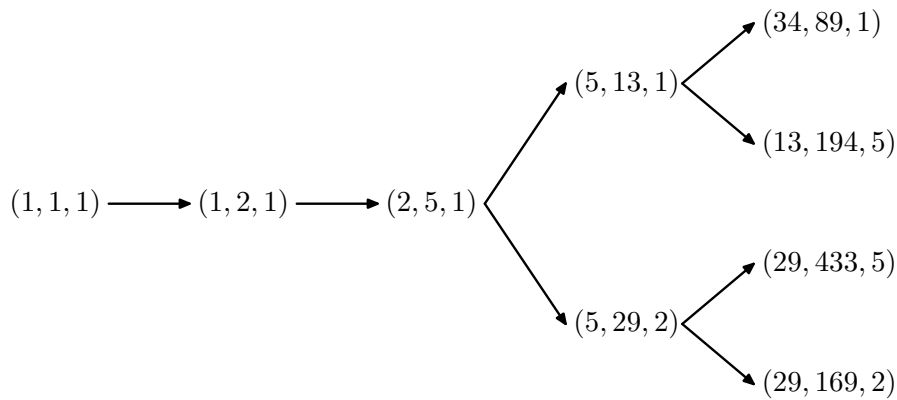}		
		\caption{The first 5 levels in the Markov tree.}
		\label{fig-mark}
\end{figure}

The following famous  Markov theorem links
the triples of the Markov tree
with the elements of the Markov spectrum below 3
by means of indefinite quadratic forms with integer coefficients.

\begin{theorem}\label{MarkovSpectrum}{\bf(A.~Markov~\cite{mar2}.)}
{\it $($i$)$}
The Markov spectrum below $3$ consists of the numbers
$\sqrt{9m^2-4}/m$, where $m$ is a positive integer such that
$$
m^2+m_1^2+m_2^2=3mm_1m_2, \qquad m_2\le m_1\le m,
$$
for some positive integers $m_1$ and $m_2$.

{
\noindent
{\it $($ii$)$} Let the triple $(m,m_1,m_2)$ fulfill the conditions of item $($i$)$.
Suppose that $u$ is the least positive residue satisfying
$$
m_2 u\equiv \pm m_1 \mod m
$$
and $v$ is defined from
$$
u^2+1=vm.
$$
Then the form
$$
f_m(x,y)=mx^2+(3m-2u)xy+(v-3u)y^2
$$
represents the value $\sqrt{9m^2-4}/m$ in the Markov spectrum.
}
\qed
\end{theorem}

The first steps to understand the phenomenon hidden in Markov's theorem were made by G.~Frobenius~\cite{frob} and R.~Remak~\cite{Remak1924}.
Following their works H.~Cohn introduced special matrices in~\cite{Cohn1955} and~\cite{Cohn1971} (which were later called {\it Cohn matrices})
whose traces are three times Markov numbers (here H.~Cohn used the
trace identity of~\cite{Fricke1896} by R.~Fricke).
As we show later (see Remark~\ref{GeneralizedMatrices}), the idea of Cohn matrices can be also extended to the case of generalised Markov and almost Markov trees, although the trace rule does not have
a straightforward generalisation.

\vspace{2mm}

{
\noindent
{\bf Main objectives of this paper.}
The generalisation of the classical Markov theory on the discrete Markov spectrum consists of the following major elements.}

\vspace{2mm}

\begin{itemize}

\item First of all we introduce a geometric approach to the classical theory. (The outline see Diagram in Figure~\ref{diagram.4}.) This approach is based on interplay between continued fractions and convex geometry of lattice points in the cones.

\vspace{1mm}

\item Basing on lattice geometry related to the classical case we construct generalised almost Markov and Markov triples of numbers (see Section~\ref{Generalised Markov triples}).

\vspace{1mm}

\item Further we relate generalised Markov triples to the elements of the Markov spectrum. We find Markov minima for the forms related to the generalised Markov trees in Corollary~\ref{(1,0)-Minima-main}.

\vspace{1mm}

\item Our next goal is to study the recursive properties of generalised Markov numbers (Corollary~\ref{main cor}). These properties are essential for fast construction of the generalised Markov tree (the classical Markov tree is constructed iteratively by the formula of Fact 2 above).

\vspace{1mm}

\item Finally we collect the main properties of the generalised Markov trees in Theorem~\ref{generaliseTheorem} (see also the diagram in Figure~\ref{diagram.5}).

\vspace{1mm}

\item We produce counterexamples for one of the generalised uniqueness conjectures, see Examples~\ref{4-11} and~\ref{4-11-2}.

\end{itemize}

\vspace{2mm}

{\noindent
{\bf Organisation of the paper.}
We start in Section~\ref{Continued fractions and lattice geometry} with the definition of continued fractions and a discussion
of lattice geometry techniques related to continued fractions.
}

\vspace{2mm}

Section~\ref{Perron Identity: general theory of Markov spectrum} is dedicated to the classical Perron Identity. Our goal here is to relate the following objects: elements of the Markov spectrum,  LLS sequences, reduced arrangements, and reduced forms.

\vspace{2mm}

Further in Section~\ref{Theory of Markov spectrum for integer forms}
we discuss the case of integer forms with integer coefficients.
In this case all the corresponding LLS sequences are periodic.
This property enriches the Perron Identity with additional interesting maps
and relations. In particular the forms are linked now with
extremal reduced $\SL(2,\z)$ matrices.

\vspace{2mm}

In Section~\ref{Triple-graphs: definitions and examples} we introduce an important general triple-graph structure which perfectly fits to Markov theory and its generalisation proposed in this paper.
After briefly defining triple-graphs in Subsection~\ref{Definition of triple-graphs}
we show several basic examples of triple-graph structure
for Farey triples, Markov triples, triples of finite sequences, and
triples of $\SL(2,\z)$-matrices
(we refer to Subsections~\ref{G-EX1}, \ref{exMarkov1}, \ref{exLLS},
and~\ref{exMarkov2} respectively).
Finally we study several important lexicographically monotone and algorithmic properties of these triple graphs.
in Subsections~\ref{TreeStruct} and~\ref{TreeAlg}.

\vspace{2mm}

We introduce the extended theory of Markov theory in Section~\ref{Markov tree and its generalisations}.
After a brief discussion of the classical case in Subsection~\ref{Classical Markov theory in one diagram}
we formulate the main definitions and discuss the extended Markov theory in
Subsections~\ref{Markov LLS triple-graphs} and~\ref{Generalised Markov triples}.
We show the diagram for the extended Markov theory in Subsection~\ref{GeneralisedSection}
(see Theorem~\ref{generaliseTheorem}).
Finally in Subsections~\ref{Uniqueness conjecture for Markov triples}
and~\ref{A counterexample to the generalised uniqueness conjecture}
we discuss the uniqueness conjecture for both classical and generalised cases.
In particular we show two counterexamples for one of the generalised triple graphs in Examples~\ref{4-11} and~\ref{4-11-2}.

We conclude this article in Section~\ref{Related theorems and proofs}
with proving all necessary statements used in the generalised Markov theorem.

\section{Continued fractions and lattice geometry}
\label{Continued fractions and lattice geometry}

In Subsection~\ref{Regular continued fractions}
we recall classical definitions of continued fraction theory.
Further in Subsection~\ref{Basics of lattice geometry}
we introduce some necessary definitions of
integer lattice geometry and describe its connection to
continued fractions.

\subsection{Regular continued fractions}
\label{Regular continued fractions}

Let us fix some standard notation for sequences and their continued fractions.

\vspace{2mm}

All sequences will be considered within parentheses.
We write
$$
\big(a_1,\ldots, a_k, \langle b_1,\ldots, b_l\rangle \big)
$$
for an eventually periodic sequence with preperiod
$(a_1,\ldots, a_k)$ and period $(b_1,\ldots, b_l)$.

\begin{definition}\label{periodisation-def}
Let $\alpha=(a_1,\ldots, a_n)$ be a finite sequence.
Denote by $\langle \alpha\rangle$
the periodic infinite sequence $\alpha=(\langle a_1,\ldots, a_n \rangle)$
with period $\alpha$. We say that $\langle \alpha \rangle$
is the {\it periodisation} of $\alpha$.
\end{definition}

We write $\alpha^n$ to replace a subsequence $\alpha\alpha\ldots\alpha$,
where $\alpha$ is repeated $n$ times.

\begin{definition}
Let $(a_1,a_2,a_3,\ldots)$ be a sequence of positive integers, except $a_1$ which
can be an arbitrary integer. The sequence can be either finite or infinite here.
The expression
$$
a_1+\frac{1}{\displaystyle a_2+\frac{1}{\displaystyle a_3+\ldots}}
$$
is called a {\it regular continued fraction} for the sequence $(a_1,a_2,a_3,\ldots)$ and dented by
$[a_1:a_2;a_3;\ldots]$.
\end{definition}

In particular $[a_1;\ldots: a_k: \langle b_1{: \ldots :} b_l\rangle]$ is a periodic continued fractions  with preperiod
$(a_1,\ldots, a_k)$ and period $(b_1,\ldots, b_l)$.

\begin{remark}
Note that for every rational number there exists a unique regular continued fraction with an odd number of elements and there exists a unique regular continued fraction with an even number of elements. For irrational numbers we have both the existence and the uniqueness of regular continued fractions.
\end{remark}

\subsection{Basics of lattice geometry}
\label{Basics of lattice geometry}

In this subsection we define basic notions of lattice geometry: integer length and integer sine. Further we introduce convex sails for integer angles
and the LLS sequence (lattice-length-sine sequence).
The LLS sequences are a complete invariant distinguishing all the different integer angles up to integer congruence. In this paper LLS sequences play the leading role in generalising of Markov numbers.

\subsubsection{Integer lengths and integer sines}

A point in $\r^2$ is called {\it integer} if its coordinates are integers.
We say that a linear transformation is {\it integer}
if it preserves the lattice of integer points.

\vspace{2mm}

A segment is called {\it integer} if its endpoints are integer points.
An angle is called {\it integer} if its vertex is an integer point.
We say that the integer angle is {\it rational}
if it contains integer points distinct to the vertex on both of its edges.
We say that an arrangement of two lines is {\it integer} if both of the lines pass through the origin.

\vspace{2mm}

An affine transformation is called integer if it preserve the lattice of integer points. The group of affine transformations
is a semidirect product of $GL(2,\z)$ and the group of all translations by integer
vectors.

\begin{definition}\label{IntCong}
Two sets are called {\it integer congruent} (or have the same {\it integer type})
if there exists an integer affine transformation taking one to the other.
\end{definition}

\begin{definition}
The {\it integer length} of an integer vector $p_1p_2$
is the index of the sublattice generated by $p_1p_2$ in the integer lattice contained in the line $p_1p_2$.
Denote it by $\il(p_1p_2)$

The {\it integer sine} of a rational integer angle $\angle ABC$ is the index of the sublattice generated by all
the points at the edges of this angles in the lattice of integer points $\z^2$.
Denote it by $\lsin(ABC)$.
\end{definition}

\begin{remark}
The integer length of an integer segment coincides with the number of interior integer points plus one.

The integer sine of an integer angle $\angle ABC$ whose edges $AB$ and $AC$ do not
contain integer points is twice the Euclidean area of the triangle $ABC$.
\end{remark}

For further information on integer trigonometry we refer to~\cite{oleg2} and
\cite{oleg4} (see also in~\cite{oleg1}).

\subsubsection{Sails and LLS sequences}

The notions of integer sine and integer length are the main ingredients to construct
a complete invariant of integer angles and integer arrangements.

\begin{definition}
Let $\angle A$ be an an integer angle with vertex at $v$.
The boundary of the convex hull of all integer points
in the interior of the $\angle A$ except $v$ is called the {\it sail} of $\angle A$.
\end{definition}

The sail is a broken line with a finite or infinite number of elements and one or two rays in it.
Let $\ldots A_{-1}A_0A_1\ldots$ be the broken line with finite segments (namely we remove any rays from the boundary).
What remains is either a finite, one-side infinite, or two-side infinite broken line.
The {\it LLS sequence} is defined as follows:
$$
\begin{array}{l}
a_{2k}=\il(A_kA_{k+1});\\
a_{2k-1}=\lsin{\angle A_{k-1}A_{k}A_{k+1}};
\end{array}
$$
for all admissible $k$.

\begin{definition}
The {\it LLS sequence} of an integer arrangement of two lines is
the LLS sequence of any of four angles that are formed by the lines of the arrangement.
\end{definition}

\subsection{Continuants and their relations to integer sines}

Let us continue with the following classical definition.

\begin{definition}
The $n$-th continuant is a {\it polynomial} of degree $n$ defined recursively:
$$
\begin{array}{l}
K_0()=1;\\
K_1(x_1)=x_1;\\
K_n(x_1,\ldots,x_n)=x_nK_{n-1}(x_1,\ldots, x_{n-1})+K_{n-2}(x_1,\ldots, x_{n-2}).
\end{array}
$$
\end{definition}

For sequences of real numbers we use the following extended definition of continuants.

\begin{definition}
Consider a sequence $\alpha=(a_1,\ldots,a_n)$ and integers $i,j$ satisfying $1\le i\le j{+}1 \le n{+}1$.
A {\it partial continuant $K_i^j$} of $\alpha$
is a real number defined as follows:
$$
K_i^j(\alpha)=K_{j-i+1}(a_i,\ldots,a_j).
$$
For simplicity we write
$$
K(\alpha)=K_1^n(\alpha)
\quad \hbox{and}
\quad
\breve K(\alpha)=K_1^{n-1}(\alpha).
$$
\end{definition}

\begin{remark}
Let $\angle A$ be an integer angle with LLS sequence $\alpha$. Then
$$
\lsin(\angle A)=K_1^n(\alpha)= K(\alpha).
$$
\end{remark}

\section{Perron Identity: general theory of Markov spectrum}
\label{Perron Identity: general theory of Markov spectrum}

In this section we study interrelations between the elements of the Markov spectrum, LLS sequences, reduced arrangements, and reduced forms.
We start in Subsections~\ref{Generic arrangements and their LLS sequences}
and~\ref{Generic reduced forms} with definitions of generic arrangements and generic reduced forms.
Further in Subsection~\ref{Marked LLS sequences and reduced generic arrangements} we relate marked LLS sequences with reduced generic arrangements.
Finally in Subsection~\ref{Perron Identity in one diagram} we formulate the original Perron Identity and rewrite it in terms of mappings between
the elements of Markov spectrum, LLS sequences, reduced arrangements, and reduced forms. We will develop further theory based on this version of the Perron Identity.

\subsection{Generic arrangements and their LLS sequences}
\label{Generic arrangements and their LLS sequences}

In this paper we mostly study the following type of
the arrangements whose lines pass through the origin
(i.e. integer arrangements).

\begin{definition}
We say that an integer arrangement of lines is {\it generic}
if its lines do not contain integer points distinct to the origin.
\end{definition}

\begin{remark}
First of all we should mention that all four angles of any generic arrangement
have the same LLS sequence which is infinite in two sides.
The adjacent angles are dual in a sense that the integers lengths for the first angle
coincide with integer sines for the second angles and vise versa
(see~\cite[Chapter~2]{oleg1} for more details).
\end{remark}

\begin{remark}
Secondly, the LLS sequence is a complete invariant of generic integer arrangements with respect to the action of $\SL(2,\z)$.
Namely, for every infinite sequence of positive integers
(taken without any direction and a starting position)
there exists a unique integer-congruence class of generic arrangements whose LLS sequences coincide with the given one
(see~\cite[Chapter~7]{oleg1} for more details).
\end{remark}

\begin{definition}
For a generic arrangement $\mathcal A$ denote by $\LLS(\mathcal A)$ the LLS sequence of $\mathcal A$.
\end{definition}

\begin{example}
In Figure~\ref{duality.1} we consider an example of an arrangement in $\r^2$.
We show four convex hulls for each of the cones in the complement to this arrangement with grey. The sails (i.e. the boundaries of these convex hulls)
are endowed with integer lengths of the segments (black digits) and with integer sines of the integer angles (white digits). They form LLS sequences
of two types:

\vspace{2mm}

\begin{itemize}
\item Two LLS sequences with a period $(1,1,2,3)$: here integer lengths are $1$ and $2$, and integer sines are $1$ and $3$.

\vspace{1mm}

\item Two LLS sequences  with a period $(3,2,1,1)$: here integer lengths are $1$ and $3$, and integer sines are $1$ and $2$.
\end{itemize}

\vspace{2mm}

In both cases the order is taken counterclockwise.
Note that the integer sines of the first type of LLS sequences are the integer lengths of the second type LLS sequence and vise versa. This is an example of the classical duality between the LLS sequences of adjacent angles
(see, e.g, in~\cite[Chapter~2]{oleg1}).

\begin{figure}
\epsfbox{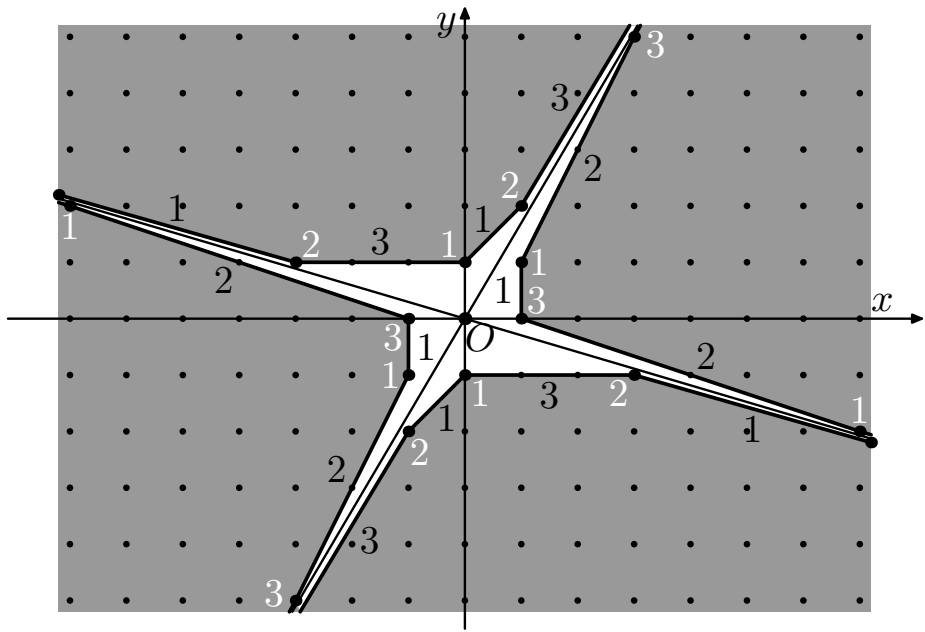}		
		\caption{An arrangement, its sails, and the LLS sequences.}
		\label{duality.1}
\end{figure}

\end{example}

We conclude this subsection with the following general remark.

\begin{remark}
The techniques of sails goes back to the original works of F.~Klein~\cite{Klein1895,Klein1896} who used sails for the generalisation of classical continued fractions to the multidimensional sail.
In fact the Klein multidimensional continued fraction seems to be an appropriate tool to study the multidimensional Markov spectrum.
Further this method was explored in more detail by V.~Arnold~\cite{Arnold2002,Arnold2003} and his school (for more details see~\cite{oleg1}).
An alternative approach was proposed by
J.H.~Conway in~\cite{Conway1997}.
Further K.~Spalding and A.P.~Veselov in~\cite{SV2018}
established a remarkable relation between Conway rivers and Klein-Arnold sails.
\end{remark}

\subsection{Generic reduced forms}
\label{Generic reduced forms}

Let us associate with every arrangement the following indefinite quadratic form.
\begin{definition}
We say that a form
$(y-px)(y-qx)$ is {\it reduced} if
$$
p> 1 \quad \hbox{and}  \quad   0>q >-1.
$$
Denote it by $f_{p,q}$.
\end{definition}

We have a similar notion of generality for quadratic forms.

\begin{definition}
Let $f$ be a quadratic form and let $r$ be a real number.
We say that $f$ {\it represents} $r$ if there exists some integer point $(x,y)\ne (0,0)$
such that $f(x,y)=r$.

\vspace{1mm}

A quadratic form $f$ is called {\it generic} if it does not represent $0$.
\end{definition}

\subsection{Marked LLS sequences and reduced generic arrangements}
\label{Marked LLS sequences and reduced generic arrangements}

Since the functional $\M$ is zero at non-generic arrangements, it remains to study the properties of $\M$ for generic arrangements.
Note also that $\M$ is constant at every $\SL(2,\z)$-orbit of integer generic arrangements.
So it is enough to choose some representatives from all $\SL(2,\z)$-orbits of integer generic arrangements.

\begin{definition}
We say that an integer generic arrangement formed by the lines $y=px$ and $y=qx$ is {\it reduced} if
$$
p> 1 \quad \hbox{and}  \quad   0>q >-1.
$$
Denote it by $\mathcal A(p,q)$.
\end{definition}

In order to relate both-side infinite sequences to generic arrangements we need to fix a starting point of a sequence and a direction, so we need the following definition.

\begin{definition}
A both side infinite sequence of numbers is said to be {\it marked} if a starting element together with a direction are chosen.
\end{definition}

Now we can bijectively associate to any arrangement its marked LLS sequences.

\begin{definition}
Consider the LLS sequence of a reduced integer arrangement
$\mathcal A=\{y=px, y=qx\}$ as above.
Let $A_0=(1,0)$ and $A_1=(1,\lfloor p \rfloor)$.
We say that the LLS sequence with the starting point $a_0=\il(A_0A_1)$ and the direction induced by the orientation of the sail from $A_0$ to $A_1$
is the {\it marked LLS sequence} for the arrangement. We denote it by $\LLS_*(\mathcal A)$.
\end{definition}

\begin{proposition}
The map $\mathcal A \to \LLS_*(\mathcal A)$ is a bijection between the set of all generic reduced arrangements
and the set of all marked infinite sequences.
\qed
\end{proposition}

Let us formulate here a general proposition, relating slopes of the lines in the arrangement and the corresponding LLS sequence
(for more details see~\cite[Chapter~3]{oleg1}).
\begin{proposition}
Consider the following two regular continued fraction expressions:
$$
\begin{array}{l}
a_+=[a_0^+;a_1:\ldots ],\\
a_-=-[a_0^-:a_{-1}:a_{-2}:\ldots ]\\
\end{array}
$$
where $a_0^\pm>0$, and at least one of them is nonzero.
Set also
$$
a_0=a_0^-+a_0^+.
$$
Then the LLS sequence of $f_{a_+,a_-}$ is
$(a_i)_{-\infty}^{+\infty}$, where $a_0$ corresponds to the segment
with endpoints $(1,a_0^{-})$ and $(1,a_0^{+})$.
\qed
\end{proposition}

For further details on arrangements and their LLS sequences we refer to
the general theory of integer trigonometry~\cite{oleg2,oleg4}
and~\cite[Chapters~3,4]{oleg1}.

\subsection{Perron Identity in one diagram}
\label{Perron Identity in one diagram}

The cornerstone of the theory of Markov Spectrum is the following
classical theorem.

\begin{theorem}\label{Perron}
Consider an indefinite binary quadratic form $f$ with positive discriminant $\Delta(f)$.
Assume that $f$ does not attain zero at integer points distinct to the origin.
Let $\mathcal A$ be the zero locus arrangement for $f$ with
$$
LLS(\mathcal A)=(\ldots a_{-2},a_{-1},a_0,a_1,a_2.\ldots)
$$
for some positive integers $a_i, i\in \z$.
Then
\begin{equation}\label{JP}
\inf\limits_{\Z^2\setminus \{(0,0)\}}\big|f\big|=\inf\limits_{i\in \Z}\bigg(\frac{\sqrt{\Delta(f)}}{a_i+[0;a_{i+1}:a_{i+2}:\ldots]+[0;a_{i-1}:a_{i-2}:\ldots]}\bigg).
\end{equation}
\end{theorem}

\begin{remark}
The expression of the theorem is known as the {\it Perron Identity}.
For simplicity Perron's theorem is expressed in the form of LLS sequences,
here we follow a recent paper~\cite{MattyOleg}.
\end{remark}

Let us rewrite the Perron Identity in the form of a diagram
(See Figure~\ref{diagram.1}).
We will extend this diagram further to visualise Markov theory and to generalise it.
Let us now describe the maps in this diagram.

\begin{figure}
$$\epsfbox{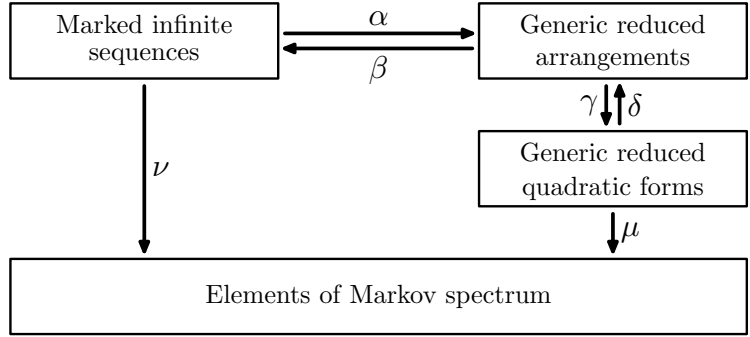}$$		
		\caption{Diagram for Perron Identity.}
		\label{diagram.1}
\end{figure}

\vspace{2mm}

{\bf Map $\alpha:$}
$$
(\ldots, a_{-2},a_{-1},a_0,a_1,a_2,\ldots)
\mapsto
\mathcal A\big([a_0,a_1,\ldots],-[0,a_{-1},a_{-2},\ldots]\big).
$$

\vspace{2mm}
{\bf Map $\beta:$}
Let $p=[a_0,a_1,\ldots]$, and $q=-[0,a_{-1},a_{-2},\ldots]$.
Then
$$
\mathcal A\big(p,q)
\mapsto
(\ldots, a_{-2},a_{-1},a_0,a_1,a_2,\ldots)
$$

\vspace{2mm}

{\bf Map $\gamma:$} Given $p,q$ set
$$
\mathcal A(p,q) \mapsto f_{p,q}.
$$

\vspace{2mm}

{\bf Map $\delta:$} Given $p,q$ set
$$
f_{p,q} \mapsto \mathcal A(p,q).
$$

\vspace{2mm}

{\bf Map $\nu$:}

$$
\begin{array}{l}
(\ldots, a_{-2},a_{-1},a_0, a_1,a_2,\ldots) \mapsto\\
\displaystyle
\qquad\qquad
\inf\limits_{i\in \Z}\Big(a_i+[0;a_{i+1}:a_{i+2}:\ldots]+[0;a_{i-1}:a_{i-2}:\ldots] \Big).
\end{array}
$$

\vspace{2mm}

{\bf Map $\mu$:}
$$
\mu: f \mapsto \inf\limits_{\Z^2\setminus \{(0,0)\}}\frac{\sqrt{\Delta(f)}}{\big|f\big|}.
$$

We conclude this section with several observation of the maps in the diagram.

\begin{remark}
It is clear that $\alpha$, $\beta$, $\gamma$, and $\delta$
are isomorphisms and in addition $\beta=\alpha^{-1}$ and $\delta=\gamma^{-1}$.
These isomorphisms give natural identifications
of the spaces of marked infinite sequences, the space of generic reduced arrangements, and the space of reduced quadratic forms.
The maps $\nu$ and $\mu$ are described by the left hand side and the right hand
side of Equation~(\ref{JP}).
The identifications provided by $\alpha$, $\beta$, $\gamma$,
and $\delta$ resulted in the Perron Identity.
\end{remark}

\begin{remark}
Map $\nu$ (and, equivalently $\mu$) does not have an inverse, as this map is not an injection.
The sequences mapping to one value are shown in Examples~\ref{4-11}
and~\ref{4-11-2} later.
\end{remark}

\begin{remark}
Since all the maps corresponding to arrows with opposite directions of the diagram in Figure~\ref{diagram.1} are inverse to each other,
the diagram is commutative. This leads to Theorem~\ref{Perron}
on the Perron Identity.
\end{remark}

\section{Theory of Markov spectrum for integer forms}
\label{Theory of Markov spectrum for integer forms}

In this section we discuss Markov theory for indefinite forms with integer coefficients.
We introduce reduced matrices and reduced forms and show their basic properties in Subsection~\ref{Reduced matrices}.
Further in Subsection~\ref{Extremal reduced forms, extremal reduced matrices, extremal sequences}
we define extremal reduced forms for which Markov minima are precisely at $(1,0)$; additionally we define extremal matrices and finite sequences
related to extremal forms.
Finally in Subsection~\ref{Theory of Markov spectrum for integer forms in one
diagram} we put together the most important relations of the theory
of the Markov spectrum for integer forms in one commutative diagram.

\subsection{Reduced matrices}\label{Reduced matrices}

In this subsection we introduce reduced matrices and reduced forms and show their multiplicative properties and relate them to the corresponding LLS sequences.

\subsubsection{Definition of reduced associated matrices and associated forms}
We start with the following general definition.

\begin{definition}
Let $(a_1,\ldots,a_n)$ be positive integers.
\begin{itemize}
\item
The matrix
$$
\left(
\begin{array}{cc}
K_2^{n-1} &K_2^{n}\\
K_1^{n-1} & K_1^n\\
\end{array}
\right)
$$
is said to be a {\it reduced} matrix {\it associated} to $(a_1,\ldots,a_n)$,
and denoted by $M_{a_1,\ldots,a_n}$.

\item
The form
$$
K_1^{n-1}x^2+(K_1^n-K_2^{n-1})xy-K_2^{n}y^2.
$$
is said to be {\it associated} to $(a_1,\ldots,a_n)$,
and denoted by $f_{a_1,\ldots,a_n}$.

\end{itemize}
\end{definition}

\subsubsection{Basic properties of reduced matrices and reduced forms}

Let us collect several important properties of reduced associated matrices and associated forms.

\begin{proposition}\label{BasicPropertiesReduced}
Let $(a_1,\ldots,a_n)$ and $(b_1,\ldots,b_m)$ be two sequences of positive integers. Then the following six statements hold.

\vspace{2mm}

{\noindent
$($i$)$ We have
$M_{a_1,\ldots,a_n}=\prod\limits_{i=1}^n M_{a_i}$.
}

\vspace{2mm}

{\noindent
$($ii$)$  We have $M_{a_1,\ldots,a_n}\cdot M_{b_1,\ldots,b_m}=M_{a_1,\ldots,a_n,b_1,\ldots,b_m}.$
}

\vspace{2mm}

{\noindent
$($iii$)$  It holds $\det M_{a_1,\ldots,a_n}=(-1)^n.$
}

\vspace{2mm}

{\noindent
$($iv$)$
The eigenlines of $M_{a_1,\ldots,a_n}$ are
$$
y=\alpha x \qquad \hbox{and} \qquad
y=\beta x
$$
where
$$
\alpha=[\langle a_1;\ldots:a_n\rangle] \qquad \hbox{and} \qquad
\beta=-[0;\langle a_n:\ldots:a_1\rangle].
$$
Therefore the LLS sequence for the corresponding arrangement is periodic
with period $(a_1,\ldots,a_n)$.
}

\vspace{2mm}

{\noindent
$($v$)$ The point $(1,0)$ is a vertex of a sail for $M_{a_1,\ldots,a_n}$.
}

\vspace{2mm}

{\noindent
$($vi$)$ The form $f_{a_1,\ldots,a_n}$ annulates
the eigenlines of $M_{a_1,\ldots,a_n}$.
}

\end{proposition}

\begin{proof}

{\it Item~$($i$)$.} We prove the statement by induction on the number of elements $n$ in the product.

{
\noindent
{\it Base of induction}. The statement is tautological for $n=1$.
}

{
\noindent
{\it Step of induction}.
Assume we prove the statement for all sequences of length $n$. Let us prove the
statement for an arbitrary sequence
$\alpha=(a_1,\ldots,a_{n+1})$.
First of all, from the definition of the continuant we get
$$
K_1^{n+1}=a_{n+1}K_1^n+K_1^{n-1}
\quad\hbox{and} \quad
K_2^{n+1}=a_{n+1}K_2^n+K_2^{n-1}.
$$

Further we have

$$
\begin{aligned}
\prod\limits_{i=1}^{n+1} M_{a_i}
={ }&
M_{a_1,\ldots,a_n}\cdot M_{a_{n+1}}
=
\left(
\begin{array}{cc}
K_2^{n-1} &K_2^{n}\\
K_1^{n-1} & K_1^n\\
\end{array}
\right)
\cdot
\left(
\begin{array}{cc}
0 &1\\
1 & a_{n+1}\\
\end{array}
\right)
\\
={ }&
\left(
\begin{array}{cc}
K_2^n &a_{n+1}K_2^n+K_2^{n-1} \\
K_1^n & a_{n+1}K_1^n+K_1^{n-1}\\
\end{array}
\right)
=
\left(
\begin{array}{cc}
K_2^{n} &K_2^{n+1}\\
K_1^n & K_1^{n+1}\\
\end{array}
\right)
\\
={ }&
M_{a_1,\ldots,a_{n+1}}.
\end{aligned}
$$
}

{\it Items~$($ii$)$ and~$($iii$)$} are straightforward corollaries of Item~$($i$)$.

\vspace{2mm}

{\it Item~$($iv$)$} This statement follows from general theory
of $\SL(2,\z)$ reduced matrices (Gauss Reduction theory) for the case of even $n$,
see e.g. in~\cite{Karpenkov2010} or in~\cite[Chapter~7]{oleg1}).

Consider now the case of a matrix $M_{a_1,\ldots,a_{2n+1}}$. Notice that
$$
M=M_{a_1,\ldots,a_{2n+1},a_1,\ldots,a_{2n+1}}=\big(M_{a_1,\ldots,a_{2n+1}}\big)^2.
$$
Therefore,
$M$ has the same eigenlines as $M_{a_1,\ldots,a_{2n+1}}$.
Now Item~$($iv$)$ for $M_{a_1,\ldots,a_{2n+1}}$
follows directly from Item~$($iv$)$ for $M$, which is true for $M$
by Gauss Reduction theory.

\vspace{2mm}

{\it Item~$($v$)$}:
Consider a triangle bounded by the eigenlines and the line $y=2x-2$.
On the one hand it contains the point $(1,0)$.
On the other hand the closure of its interior does not contain integer points other than $(1,0)$
due to the explicit expression of Item~$($iv$)$: for eigenlines.

Direct computation shows {\it Item~$($vi$)$}.
\end{proof}

\begin{remark}\label{GeneralizedMatrices}
In this paper we use
$\SL(2,\z)$-reduced matrices which were used in so called Gauss Reduction theory (for more details
see~\cite{Katok2003,Manin2002,Lewis1997,Karpenkov2010}
and also in~\cite[Section~7]{oleg1}).
We should note that there is an alternative choice of matrices
for which the main statements have a straightforward translation.
The reduced matrices and {\it alternative matrices} are as follows:
$$
\left(
\begin{array}{cc}
K_2^{n-1} &K_2^{n}\\
K_1^{n-1} & K_1^n\\
\end{array}
\right)
\quad
\hbox{and}
\quad
\left(
\begin{array}{cc}
K_1^{n} &K_1^{n-1}\\
K_2^{n} & K_2^{n-1}\\
\end{array}
\right).
$$
Here one matrix is obtained from another by a swap of $x$ and $y$ coordinates.

It remains to mention here that the theory of these matrices follow Cohn matrices for the classical Markov case, where the corresponding tree is generated,
for instance, by the following two Cohn matrices
$$
M=\left(
\begin{array}{cc}
1 &1\\
1 &2\\
\end{array}
\right)
\quad
\hbox{and}
\quad
N=\left(
\begin{array}{cc}
1 &2\\
2 &5\\
\end{array}
\right).
$$
(In the notation of Subsection~\ref{exMarkov2} below the corresponding
triple-graph is $\G_{M,N}(\SL(2,\z),\bullet).$)
Recall that Cohn matrices were introduced in~\cite{Cohn1955,Cohn1971} by H.~Cohn. They were used for the study of Markov numbers
based on works~\cite{frob} and~\cite{Remak1924}.
\end{remark}

\begin{remark}
It is interesting to note that any $\SL(2,\z)$ matrix
$$
\left(
\begin{array}{cc}
a&c\\
b&d\\
\end{array}
\right)
$$
that satisfies
$$
d>b\ge a \ge 0
$$
is reduced. It is then equivalent to the product of even number of matrices of
elementary type $M_{a_i}$.
\end{remark}

\subsection{Extremal reduced forms, extremal reduced matrices, extremal even sequences}
\label{Extremal reduced forms, extremal reduced matrices, extremal sequences}

We start with the following general remark.

\begin{remark}\label{even}
As we have seen from Proposition~\ref{BasicPropertiesReduced}$($iii$)$
a matrix $M_{a_1,\ldots,a_n}$ belongs to  $\SL(2,\z)$ if
and only if $n$ is even.
Further we restrict ourselves entirely to the case of $\SL(2,\z)$ matrices,
and therefore we study the case of even finite sequences.
\end{remark}

Let us finally give the following definition.

\begin{definition}\label{extremal}
$\quad$

\begin{itemize}

\item
We say that a finite even sequence of integers is {\it extremal}
if the associated form
attains its normalised Markov minimum at point $(1,0)$.

\item
An $\SL(2,\z)$ matrix associated to an extremal sequence is called {\it extremal}.

\item
A form associated to an extremal sequence is called {\it extremal}.
\end{itemize}
\end{definition}

\subsection{Theory of Markov spectrum for integer forms in one
diagram}
\label{Theory of Markov spectrum for integer forms in one
diagram}

\begin{figure}
$$\epsfbox{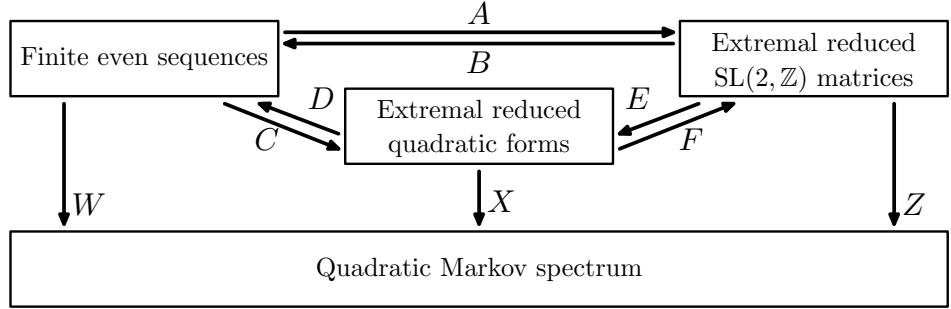}$$
		\caption{Sequences, forms, matrices, and Markov spectrum.}
		\label{diagram.2}
\end{figure}

In Figure~\ref{diagram.2} we show a diagram
that gives bijections between the set of all finite sequences,
extremal reduced matrices,
and extremal reduced forms (Maps~$A$-$F$).
Additionally we have mappings to the Markov spectrum (Maps~$W$,
$X$, and~$Z$). Let us describe these maps in more details.

\vspace{2mm}

{\bf Map $A$:} Let $(a_1,\ldots,a_n)$ be positive integers (here $n$ is assumed to be even). Then

\begin{equation}\label{MapA}
( a_1,\ldots,a_n )
\mapsto
M_{a_1,\ldots,a_n}=
\left(
\begin{array}{cc}
K_2^{n-1} & K_2^{n}\\
K_1^{n-1} & K_1^n\\
\end{array}
\right).
\end{equation}

{\bf Map $B$:}
For $d>b>a\ge 0$ (and therefore $c=\frac{ad-1}{b}$) and the corresponding reduced matrix we have:
\begin{equation}\label{MapB}
\left(
\begin{array}{cc}
a & c\\
b & d\\
\end{array}
\right)
\mapsto
\Big\langle a_1,\ldots, a_{2n-1},\Big\lfloor\frac{d{-}1}{b}\Big\rfloor\Big\rangle.
\end{equation}
Here $[a_1;a_2:\ldots:a_{2n-1}]$ is the regular odd continued fraction decomposition
for $b/a$.

\begin{proposition}
The maps $A$ and $B$ are inverse to each other.
\qed
\end{proposition}

In fact the map $A$ can be extended analytically (i.e. with the same formula)
to a bijection between the set of all finite sequences and the set of all reduced operators.
This bijection delivers a complete invariant of $\SL(2,\z)$ conjugacy
classes of $\SL(2,\z)$ matrices.
A slightly modified version of this approach is known as
Gauss Reduction theory
(for more details see~\cite{Lewis1997,Manin2002,Katok2003,Karpenkov2010} and~\cite[Chapter~7]{oleg1}).

\vspace{2mm}

{\bf Map $C$:}
For a finite sequence $(a_1,\ldots,a_n)$ we set
$$
(a_1,\ldots,a_n)
\mapsto
f_{a_1,\ldots,a_n}(x,y)=
K_1^{n-1}x^2+(K_1^n-K_2^{n-1})xy-K_2^{n}y^2.
$$

\vspace{2mm}

{\bf Map $D$:}
We do not have a nice explicit form for this map.
The map is a composition $B\circ F$. (See map $F$ below).

\vspace{2mm}

{\bf Map $E$:} We have the following simple formula here:
$$
\left(
\begin{array}{cc}
a & c\\
b & d\\
\end{array}
\right)
\mapsto
bx^2+(d-a)xy-cy^2.
$$

\vspace{2mm}

{\bf Map $F$:}
Here we have
$$
Ax^2+Bxy+Cy^2
\mapsto
\left(
\begin{array}{cc}
a & c\\
b & d\\
\end{array}
\right)
$$
where
$$
a=\frac{-B+\sqrt{B^2-4AC+4}}{2},
\qquad
b=A,
\qquad
c=-C,
\qquad
d=a+B.
$$

\vspace{2mm}

{\bf Map $W$:}
For this map we have a nice expression in terms of
continuants:$$
LLS
\mapsto
\frac{\sqrt{(K_1^n+K_2^{n-1})^2-4}}{K_1^{n-1}}.
$$
This is equivalent to applying Map $\nu$ in the diagram of Figure~\ref{diagram.1} to the periodisation
of a sequence.

\vspace{2mm}

{\bf Map $X$:} This is a restriction of Map~$\mu$ above.
In addition we have the following useful formula:
$$
f \mapsto \frac{\sqrt{\Delta(f)}}{f(1,0)}.
$$

{\bf Map $Z$:} This map is defined as follows:
$$
\left(
\begin{array}{cc}
a & c\\
b & d\\
\end{array}
\right)
\mapsto
\frac{\sqrt{(a+d)^2-4}}{b}.
$$

\begin{remark}
Since all the maps corresponding to arrows with opposite directions of the diagram in Figure~\ref{diagram.2} are inverse to each other,
the diagram is commutative, namely it holds for single elements, see
for instance Figure~\ref{diagram.3} of Example~\ref{exxx}.
\end{remark}

Here and below to avoid ubiquity for a one element sequence $(a)$ we write
$(\underline{a})$.

\begin{example}\label{exxx}
Let us give several examples of various
periods of LLS sequences together with corresponding matrices, forms, and Markov elements.

	\begin{center}
		\begin{tabular}{ |c|c|c| }
			\hline
			Sequence & Matrix / Form &  Markov Element  \\ \hline
$(p,p)$ &
$
\begin{array}{c}
\left(\begin{matrix}
				1 & p \\
				p & p^2+1 \end{matrix}\right)\\
px^2-p^2xy-py^2
\end{array}
$
&
$
\displaystyle
\frac{2}{[\langle p\rangle]}
$
\\
\hline
$(2,2,3,3)$ &
$
\begin{array}{c}
\left(\begin{matrix}
			7 & 23 \\
			17 & 56 \end{matrix}\right)\\
17x^2+49xy-23y^2
\end{array}
$
&
$
\displaystyle
\frac{\sqrt{3965}}{17}
$
\\
\hline
$(2,2,3,3,3,3)$ &
$
\begin{array}{c}
\left(\begin{matrix}
			76 & 251 \\
			185 & 611 \end{matrix}\right)\\
185x^2+535xy-251y^2
\end{array}
$
&
$
\displaystyle
\frac{\sqrt{471965}}{185}
$
\\
\hline
$(2,2,2,2,3,3)$ &
$
\begin{array}{c}
\left(\begin{matrix}
			41 & 135 \\
			99  & 326 \end{matrix}\right)\\
99x^2+285xy-135y^2
\end{array}
$
&
$
\displaystyle
\frac{\sqrt{14965}}{33}
$
\\
\hline
$
\begin{array}{l}
(2,2,3,3,\\
\hphantom{(}2,2,3,3,3,3)
\end{array}
$
&
$
\begin{array}{c}
\left(\begin{matrix}
			4787 & 15810 \\
			11652 & 38483 \end{matrix}\right)\\
11652x^2+33696xy-15810y^2
\end{array}
$
&
$
\displaystyle
\frac{\sqrt{13002034}}{971}
$
\\ \hline
		\end{tabular}
	\end{center}

On the following diagram we consider a single case of the sequence $(2,2,3,3,3,3)$
and its corresponding images under the above maps (see Figure~\ref{diagram.3}).

\begin{figure}
$$\epsfbox{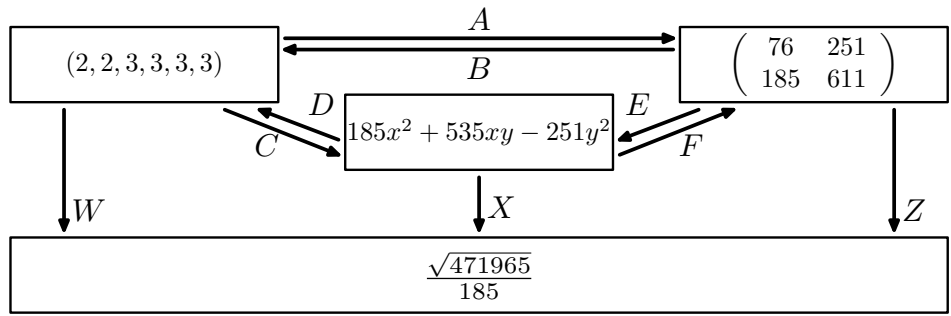}$$
		\caption{An example for $(2,2,3,3,3,3)$.}
		\label{diagram.3}
\end{figure}

\end{example}

Let us finally note that this section covers forms
for all possible periodic arrangements due to the following proposition.

\begin{proposition}
$($i$)$ After an appropriate $SL(2,\z)$ change of coordinates
any form corresponding to a periodic arrangement
is multiple to some extremal reduced quadratic form.

{\noindent
$($ii$)$
Let $A$ be an $\SL(2,\z)$ matrix.
Then either $A$ or $-A$ is conjugate to some extremal form.
}
\end{proposition}

\begin{proof}
{\it Item $($i$)$}.
Any form with a periodic arrangement determines the periodic LLS sequence $\alpha$
obtained by the compositions of Map~$\delta$ and Map~$\alpha$ of Figure~\ref{diagram.1}.
Take an even period of this LLS sequence and apply to it the map~$C$.
We get an extremal reduced quadratic form multiple to the original one.

{\it Item $($ii$)$}.
The second statement is equivalent to Gauss Reduction theory in dimension 2 (see~\cite[Chapter~7]{oleg1}).
\end{proof}

\section{Triple-graphs: definitions and examples}
\label{Triple-graphs: definitions and examples}

In this section we introduce a supplementary abstract structure of triple-graphs which is suitable for the study of generalised Markov theory.
We start in Subsection~\ref{Definition of triple-graphs} with general notions and definitions.
In Subsection~\ref{G-EX1} we define Farey triple-graph structure, which
leads to the definition of a natural Farey coordinate for
arbitrary Farey graphs.
Further in Subsections~\ref{exMarkov1}, \ref{exLLS},
and~\ref{exMarkov2}  we introduce a natural tree-graph structure on finite Markov numbers, integer sequences and $\SL(2,\z)$-matrices respectively. For Markov numbers and their related matrices, in the classical case the triple-graphs are similar to the ones studied in the book~\cite{Aigner1989} by M.~Aigner.
In Subsection~\ref{TreeStruct} we discuss conditions for the above triple-graphs to have a tree structure.
Finally we discuss how to reconstruct triples of sequences in triple-graphs
by their central elements in Subsection~\ref{TreeAlg}.

\subsection{Definition of triple-graphs}
\label{Definition of triple-graphs}

In this subsection we introduce a general triple-graph structure.
Let $S$ be an arbitrary set, and
$$
\sigma:S^3 \to S
$$
be a ternary operation on it, where $S^3=S\times S \times S$.
Set
$$
\begin{array}{l}
L_\sigma(a,b,c)=\big(a,\sigma(a,b,c),b\big),\\
R_\sigma(a,b,c)=\big(b,\sigma(b,c,a),c\big).\\
\end{array}
$$

\begin{definition}
Let $S$ be an arbitrary set, and let $\sigma$ be a ternary operation on it.
Denote by $\G(S,\sigma)$ the directed graph whose vertices are elements of $S^3$.
The vertices $v,w\in S^3$ are connected by an edge $(v,w)$ if either
$$
w=L_\sigma(v), \qquad \hbox{or} \qquad w=R_\sigma(v).
$$
\end{definition}

\begin{definition}
For a vertex $v$ in $\G(S,\sigma)$
denote by $\G_v(S,\sigma)$ the connected component of $\G(S,\sigma)$ that contains $v$.
\end{definition}

\begin{definition}
Any element $w\in\G_v(S,\sigma)$ can be written as
\begin{equation}\label{FarCod}
w=R^{a_{2n}}\circ L^{a_{2n-1}}\circ\ldots \circ R^{a_2}\circ
L^{a_1}(v),
\end{equation}
where $a_2,\ldots, a_{2n-1}$ are positive integers and $a_1$, $a_{2n}$ are nonnegative integers.
We say that
$$
(a_1,a_2,\ldots,a_{2n})
$$
is a {\it Farey code} of $w$ in $G_v(S,\sigma)$, denote it by $F(w)$.
\\
We say that the continued fraction
$$
[0;a_0+1:a_1:\ldots:a_{2n}+1]
$$
is a {\it Farey coordinate} of $w$, denote it by $w_F$.
\end{definition}

\begin{definition}
We say that a graph $\G_v(S,\sigma)$ is {\it free generated} if
every $w$ in $G_v(S,\sigma)$ has a unique Farey coordinate
(or in other words, the representation in Expression~(\ref{FarCod}) is unique
for every $w$ in $G_v(S,\sigma)$).
\end{definition}

In case if $\G_v(S,\sigma)$ is free generated,
the graph $\G_v(S,\sigma)$ is a binary rooted tree.
Every element has a unique Farey code, and a unique Farey coordinate.
Farey coordinates cover all rational numbers of the open interval $(0,1)$.

\begin{definition}
Let $\G_v(S_1,\sigma_1)$ and $\G_v(S_2,\sigma_2)$  be two triple-graphs.
We say that an isomorphism of $S_1$ and $S_2$
is an {\it interior isomorphism} of triple-graphs.
\end{definition}


\subsection{Farey triples, Farey coordinates for triple-graph structure}
\label{G-EX1}
We start with a simple example of a triple-graph structure generated by Farey triples.
It is related to triangular chambers in Farey tessellation in hyperbolic geometry.

\vspace{2mm}

First of all we define Farey summation.
\begin{definition}
Let $\frac{p}{q}$ and $\frac{r}{s}$ be two rational numbers with $q,s>0$ and such that $\gcd(p,q)=\gcd(r,s)=1$.
Then
$$
\frac{p}{q}\hat\oplus \frac{r}{s}=
\frac{p+r}{q+s}.
$$
\end{definition}

For a triple of rational numbers we set
$$
\hat\oplus(r_1,r_2,r_3)=r_1\hat\oplus \, r_2.
$$

We have a triple-graph $\G_{(0/1,1/2,1)}(\q,\hat\oplus)$.
It turns out that this graph is free generated.

\begin{proposition}
Let $v$ be a triple of $\G_{(0/1,1/2,1)}(\q,\hat\oplus)$. Then its middle
element coincides with its Farey coordinate $v_F$.
\qed
\end{proposition}

In fact triples here are precisely the rational numbers in the vertices of triangles in Farey Tessellation
(for further information we refer to~\cite[Section 23.2]{oleg1}).

\subsection{Markov triple-graph}
\label{exMarkov1}

Consider the following ternary operation on the set of integers:
$$
\Sigma(a,b,c)=3ac-b.
$$
We call this operation {\it Markov multiplication} of triples.

\vspace{2mm}

By construction we have the following statement.

\begin{theorem}
The set of Markov triples coincides with the set of all triples obtained by permutation of
elements in the vertices of $\G_{(1,1,1)}(\z,\M)$.
\qed
\end{theorem}


Further we mostly consider a free generated part of it:  $\G_{(1,5,2)}(\z,\Sigma)$.
Note that the complement to the set of vertices of $\G_{(1,1,1)}(\z,\M)$
for the set of vertices of $\G_{(1,5,2)}(\z,\M)$ is
$$
\big\{(1,1,1),(1,2,1)\big\}.
$$
By Theorem~3.3 in~\cite{Aigner1989}, the triple-graph $\G_{(1,5,2)}(\z,\Sigma)$ is free generated.

\subsection{Triple-graphs of finite sequences}
\label{exLLS}
Let $\z^\infty$ be the set of finite sequences with integer elements.
Consider a binary operation $\oplus$ on $\z^\infty$ defined as
$$
(a_1,\ldots,a_n) \oplus (b_1,\ldots b_m)=
(a_1,\ldots,a_n,b_1,\ldots b_m).
$$
Finally for $a,b,c\in \z^{\infty}$ set
$$
\oplus(a,b,c)=a\oplus b.
$$
Consider the triple-graph
$\z^{\infty}(a,b)= \G_{(a,a\oplus b,b)}(\z^\infty,\oplus)$.
This graph is free generated if and only if
$a$ and $b$ are not multiple to the same sequence $c$.
Here we say that $p$ is multiple to $q$ if there exists an integer $n$ such that
$$
p=\oplus_{i=1}^n q.
$$

\subsection{Triple-graphs of matrices}
\label{exMarkov2}

Let $A$, $B$, and $C$ be $\SL(2,\z)$ matrices.
Set
$$
\bullet(A,B,C)=AB.
$$

Now one can consider a triple-graph of matrices
$$
\G_{(A,AB,B)}(\SL(2,\z),\bullet).
$$

\begin{definition}
Let $M$ and $N$ be two $\SL(2,\z)$ matrices.
Denote the triple graph $\G_{(M,MN,N)}(\SL(2,\z),\bullet)$
by $\G_\bullet(M,N)$.
\end{definition}

\subsection{Tree structure of $\G_\oplus(\mu,\nu)$ and $\G_\bullet\big(M_\mu,M_\nu\big)$}
\label{TreeStruct}

In this subsection we investigate when triple-graphs $\G_\oplus(\mu,\nu)$ and $\G_\bullet\big(M_\mu,M_\nu\big)$ has a tree structure.

\subsubsection{A skew-lexicographical order}
Let us introduce a standard order on the set of all finite and infinite sequences.
\begin{definition}
Let $\alpha=(p_i)_{i=1}^{\infty}$ and $\beta=(q_i)_{i=1}^{\infty}$
be two infinite sequences of real numbers.
We write $\alpha\succ\beta$ if there exists $n$ such that
\[
\begin{cases}
p_i=q_i, & i=0,\ldots, n-1;\\
p_{n}>q_{n}, & \hbox{if $n$ is odd};\\
p_{n}<q_{n}, & \hbox{if $n$ is even}.
\end{cases}
\]

If $\alpha$ and $\beta$ coincide, then we write $\alpha=\beta$.
In all the other cases we write $\alpha\prec\beta$.

Such ordering is called {\it skew-lexicographic}.
\end{definition}

Recall that for a finite or infinite sequence of positive integers $\alpha=(a_1,a_2,\ldots)$
we denote by $[\alpha]$ the following real number
$$
[\alpha]=[a_1:a_2;\ldots].
$$

From the general theory of regular continued fractions we have the following statement.
\begin{proposition}\label{LastLabel169}
Let $\mu$ and $\nu$ be two infinite sequences of positive integers.
Then $\mu\succ \nu$ if and only if
$
[\mu] > [\nu].
$
\qed
\end{proposition}

\subsubsection{Evenly-composite sequences}
In what follows we use the following general definition.

\begin{definition}
Let $\alpha$ be a finite even sequence.
We say that $\alpha$ is {\it evenly-composite} if there exists an even sequence $\beta$
such that
$$
\alpha=\beta\oplus\beta\oplus\ldots\oplus\beta.
$$
Otherwise we say that $\alpha$ is {\it evenly-prime}.
\end{definition}

\subsubsection{A skew-lexicographical order for concatenation of even sequences}

Now we prove a rather important proposition on skew-lexi\-co\-graphic order for concatenations.

\begin{proposition}\label{proposition1-lex}
Let $\alpha$ and $\beta$ be evenly-prime sequences of integers.
Assume that $\langle\alpha\rangle \prec\langle\beta\rangle$ then we have
$$
\begin{array}{l}
\hbox{$($i$)$\hphantom{i}}\qquad
\langle\alpha\rangle \prec\langle\alpha\oplus \beta\rangle \prec \langle \beta \rangle;
\\
\hbox{$($ii$)$}\qquad
\langle\alpha\rangle \prec\langle\beta\oplus\alpha\rangle \prec \langle \beta \rangle.
\end{array}
$$
\end{proposition}

\begin{remark}
Note that the skew lexicographic order for periodisation of infinite sequences might be different
from the skew-lexicographic order (naturally extended to finite sequences) for the sequences themselves. For instance,
in the case of $\alpha=(1,1,2,2,1,1)$ and $\beta=(1,1,2,2)$, then $\beta$ follows $\alpha$,
but $\langle \beta\rangle \prec \langle \alpha\rangle$.
\end{remark}

\begin{proof}[Proof of Proposition~\ref{proposition1-lex}]
{\it Item~$($i.2$).$}
First of all we prove that $\langle\alpha\oplus \beta\rangle \prec \langle \beta \rangle$.

The proof is split into the following three cases.

{\noindent
{\it Case 1.}
Let $\beta=\alpha^n\hat\alpha$, where $n\ge 1$ and $\hat \alpha$ is an even
sequence such that $\alpha=\hat\alpha\oplus \alpha_0$ for some non-empty sequence $\alpha_0$.
}

By assumption we have
$$
[\langle\alpha\rangle]<[\langle\alpha^n\hat \alpha\rangle].
$$
Hence by cancelling the first $n$ copies of $\alpha$ (and since $\alpha$ consists of an even
number of elements) we have:
$$
[\langle\alpha\rangle]<[\langle\hat \alpha\alpha^n\rangle].
$$
Let us now compare the first several elements from both sides. We have
$$
[\alpha\hat\alpha] =
[\hat\alpha\alpha_0\hat\alpha] \le
[\hat\alpha\hat\alpha\alpha_0]=
[\hat\alpha\alpha].
$$
Here the middle inequality is equivalent to (since $\hat \alpha$ is even)
$$
[\alpha_0\hat\alpha] \le
[\hat\alpha\alpha_0].
$$
Now if equality holds then $\alpha$ has a nontrivial even isomorphism,
and therefore it is not evenly-prime. Hence we have a strict inequality:
$$
[\alpha\hat\alpha]
<
[\hat\alpha\alpha].
$$

This implies that
$$
[\alpha^{n+1}\hat\alpha]
<
[\alpha^{n}\hat\alpha\alpha].
$$
These are even partial fractions for
$[\alpha\beta]$ and $[\beta]$ respectively.
Hence
$$
[\alpha\beta]<[\beta].
$$
We have completed the proof for Case~1.

\vspace{2mm}

{\noindent
{\it Case 2.}
Now let $\beta=\alpha^n\hat\alpha\gamma$,
where $\alpha<\hat\alpha\gamma$, and difference is assumed at the very first element of $\gamma$. Here $n\ge 0$, and $\hat \alpha$ is as in Case~1.
}

By assumption we have $[\langle\alpha\rangle]<[\langle\beta\rangle]$ and hence
$$
[\langle\alpha\rangle]<[\langle\alpha^n\hat\alpha\gamma\rangle]
$$
Hence
$$
[\alpha\beta]=
[\alpha^{n+1}\hat\alpha\gamma]=
[\alpha^{n}\alpha\hat\alpha\gamma]<
[\alpha^{n}\hat\alpha\gamma]
=[\beta],
$$
notice that here we manipulate finite continued fractions. The inequality holds by the assumption.

The last inequality implies
$$
[\langle\alpha\beta\rangle]<[\langle\beta\rangle],
$$
as this is true at the element with which $\gamma$ starts inside $\beta$.

\vspace{2mm}

{\noindent
{\it Case 3.}
Finally it remains to show the case $\beta=\hat \alpha\ne \alpha$ ($n=0$ of Case~1).
We rewrite it as
$\alpha=\beta^m\hat\beta$, where $m>0$ and $\hat \beta$ is a beginning of $\beta$
with an even number of elements.
}

By assumption we have
$$
[\langle\alpha\rangle]=
[\langle\beta^n\hat\beta\rangle]<
[\langle\beta\rangle].
$$
Therefore, after cancelling $\beta^n$ at the beginning (since $\beta$ has an even number of elements)
we have
$$
[\langle\hat\beta\beta^n\rangle]<
[\langle\beta\rangle].
$$
In particular
$$
[\hat \beta\beta]
\le
[\beta\hat \beta].
$$
As in Case~1 (since $\beta$ is evenly prime) the last inequality is strict.
Therefore, we have
$$
[\langle\alpha\beta\rangle]=
[\langle\beta^n\hat\beta\beta\rangle]<
[\langle\beta\rangle].
$$

This concludes the proof for
$
[\langle\alpha\beta\rangle]< [\langle\beta\rangle].
$

\vspace{2mm}

{\noindent
{\it Item~$($i.1$).$}
The case $\langle\alpha\rangle\prec\langle\alpha\beta\rangle$ after cancelation of one $\alpha$
is equivalent to
$$
\langle\alpha\rangle\prec\langle\beta\alpha\rangle.
$$
Let us rewrite it as
$$
\langle\beta\alpha\rangle\succ\langle\alpha\rangle.
$$
Now the proof repeats the proof given above with all symbols '$<$' and '$\le$'
changed to '$>$' and '$\ge$' in all inequalities respectively.
}

\vspace{2mm}

{
\noindent
{\it Item~$($ii$).$}
Finally the proof for Item~{\it $($ii$)$} repeats the proof for
Item~{\it $($i$)$}  (for both statements)
with all symbols '$<$' and '$\le$'
changed to '$>$' and '$\ge$' in all inequalities respectively
and vice versa.
}
\end{proof}

\subsubsection{Skew-lexicographic order for $\G_\oplus(\mu,\nu)$}
We start with the following definition.

\begin{definition}
We say that a path $v_1,\ldots, v_n$ in a triple-graph $\G$ is {\it descending}
if for every $i=1,\ldots,n{-}1$ we have
$$
v_{i+1}=L_\sigma(v_i) \qquad \hbox{or} \qquad v_{i+1}=R_\sigma(v_i).
$$
\end{definition}

Let $\alpha$ be a finite sequence. Recall that the periodisation $\langle \alpha\rangle$
is a periodic infinite sequence with period $\alpha$ (see Definition~\ref{periodisation-def}).

\begin{proposition}\label{periodisation}
Let $\mu$ and $\nu$ be two evenly-prime sequences of integers.
Assume that $\langle\mu\rangle\prec\langle\nu\rangle$ then we have
the periodisations of middle elements in the triple-graph
$\G_\oplus(\mu,\nu)$
skew-lexicographically increasing with respect to the Farey coordinate.
Namely, let
$$
(\alpha_i,\beta_i,\gamma_i)\in \G_\oplus(\mu,\nu)
\quad \hbox{for $i=1,2$}
$$
with Farey coordinates $w^i_F$ for $i=1,2$ respectively.
Then
$$
w^1_F<w^2_F \qquad \hbox{implies} \qquad
\langle\beta_1\rangle \prec \langle\beta_2\rangle.
$$
\end{proposition}

\begin{remark}
Here we do not assume that the Farey coordinate is uniquely
defined for the vertices of the triple-graph.
We get the uniqueness of Farey coordinate later in Corollary~\ref{FareyUni}.
\end{remark}

\begin{lemma}\label{lemma2-lex}
Let $v_1,\ldots, v_n$ be a descending path in a triple-graph
$\G_\oplus(\mu,\nu)$ with $\langle\mu\rangle \prec \langle\nu\rangle$.
Let also $v_i=(\alpha_i,\beta_i,\gamma_i)$ for $i=1,\ldots,n$.
Then
$$
\langle\alpha_1\rangle
\prec
\langle\beta_n\rangle
\prec
(\gamma_1).
$$
\end{lemma}

\begin{proof}
For $n=1$ the statement follows directly from Proposition~\ref{proposition1-lex} and
$$
v_1=(\mu,\mu\oplus\nu,\nu).
$$

\vspace{2mm}

By the  induction on the number of operations $R$ and $L$ and by Proposition~\ref{proposition1-lex} for every triple
$$
(\alpha,\beta,\gamma)
\in\G_\oplus(\mu,\nu)
$$
we have:
$$
\langle\alpha\rangle
\prec
\langle\beta\rangle
\prec
\langle\gamma\rangle.
$$

\vspace{2mm}

The last directly implies the statement of Lemma~\ref{lemma2-lex}.
\end{proof}

{
\noindent
{\it Proof of Proposition~\ref{periodisation}}.
Let $(\alpha_p,\beta_p,\gamma_p)$
and $(\alpha_q,\beta_q,\gamma_q)$
be two triples of $\G_\oplus(\mu,\nu)$ whose Farey coordinates satisfy
$$
(\alpha_p,\beta_p,\gamma_p)_F<(\alpha_q,\beta_q,\gamma_q)_F.
$$
Consider the shortest path in the Farey tree
connecting these two Farey coordinates.
Let now the Farey coordinate of
$$
(\alpha_c,\beta_c,\gamma_c)\in\G_\oplus(\mu,\nu)
$$
have the earliest tree level in this path.
}

\vspace{2mm}

Consider two descending paths in $\G_\oplus(\mu,\nu)$
that correspond to the shortest paths in the tree 
from the $c$-triple to the $p$-triple and from the $c$-triple to the $q$-triple.
Denote them by $v_1,\ldots v_n$ and $w_1,\ldots w_m$ respectively.

\vspace{2mm}

It is clear that $v_1=w_1=(\alpha_c,\beta_c,\gamma_c)$ , and that
$$
v_2=L_\oplus (\alpha_c,\beta_c,\gamma_c)
\qquad  \hbox{and} \qquad
w_2=R_\oplus (\alpha_c,\beta_c,\gamma_c).
$$

Now on the one hand, by Lemma~\ref{lemma2-lex} applied to the path $v_2,\ldots, v_n$ we have
$$
\langle\beta_p\rangle
\prec
\langle\beta_c\rangle
$$
since $\beta_c$ is the third element in the triple $v_2$.

\vspace{2mm}

On the other hand, by Lemma~\ref{lemma2-lex} applied to the path $w_2,\ldots, w_m$ we have
$$
\langle\beta_c\rangle
\prec
\langle\beta_q\rangle
$$
since $\beta_c$ is the first element in the triple $w_2$.

\vspace{2mm}

Therefore,
$$
\langle\beta_p\rangle
\prec
\langle\beta_c\rangle
\prec
\langle\beta_q\rangle.
$$

\qed

\begin{corollary}\label{FareyUni}
Let $\mu$ and $\nu$ be two evenly-prime sequences of integers.
Assume $\langle\mu\rangle \prec \langle\nu\rangle$.
Then the triple-graph $\G_\oplus(\mu,\nu)$ is a tree.
(In the other words, Farey coordinate is a complete invariant of vertices.)
\qed
\end{corollary}

The last corollary can be reformulated for Generalised Gauss-Cohn
matrices due to the fact that Map $A$ has an inverse (see Equations~(\ref{MapA}) and~(\ref{MapB})):

\begin{corollary}
Let $\mu$ and $\nu$ be two evenly-prime sequences of integers.
Assume $\langle\mu\rangle \prec \langle\nu\rangle$.
Then the triple-graph $\G_\bullet(M_\mu,M_\nu)$ is a tree.
\end{corollary}

\subsection{Reconstruction of triples in the triple-graphs
by their central elements}
\label{TreeAlg}

The reconstruction of sequence triples in the triple-graphs by their central elements
is based on the following simple proposition.

\begin{proposition}\label{once-algorithm}
Let $\mu$ and $\nu$ be two evenly-prime sequences of integers
satisfying $\langle\mu\rangle \prec \langle\nu\rangle$.
Suppose that a finite sequence $\alpha\notin \{\mu,\nu\}$
appears in some triple in the triple-graph
$\G_\oplus(\mu,\nu)$.
Then $\alpha$ appears exactly once at the center of a triple.
\end{proposition}

\begin{proof}
Consider $\alpha\notin \{\mu,\nu\}$.
Let $\alpha$ be an element of some triple in $\G_\oplus(\mu,\nu)$.
Therefore $\alpha$ is constructed
(i.e. appears at the highest possible level)
as a concatenation of some other two sequences
$\beta$ and $\gamma$ in some triple
$$
(\beta,\alpha,\gamma)\in \G_\oplus(\mu,\nu).
$$
This implies the existence.

The uniqueness follows directly from Proposition~\ref{periodisation}.
\end{proof}

\begin{remark}\label{Triple-reconstruction}
Let us briefly mention that any triple of $\G_\oplus(\mu,\nu)$
is reconstructible from its middle element.

From Proposition~\ref{once-algorithm} we have a one-to-one correspondence between central elements and triples.
So one can search the necessary triple by a brute force algorithm
that studies all the triples level by level
in the tree $\G_\oplus(\mu,\nu)$.

It remains to note that the middle LLS sequences in triples at level $n$ have at least $2n$ elements.
Therefore the proposed brute force algorithm will stop in a time exponential with respect to the length of the central sequence.
\end{remark}

\begin{remark}
As for Markov triples,
a similar brute force algorithm finds in a finite time all possible triples with a given central element.
However the uniqueness of a triple
with the prescribed central element
is equivalent to the uniqueness conjecture (see~Conjecture~\ref{Uniqueness-conjecture} below) formulated
by G.~Frobenius in 1913, which is still open now.
\end{remark}

\section{Markov tree and its generalisations}
\label{Markov tree and its generalisations}

In this section we discuss a generalisation of the Markov tree.
First of all we reformulate a classical theorem in our settings and write the diagram for it in Subsection~\ref{Classical Markov theory in one diagram}.
Further we define Markov LLS triple-graphs in~\ref{Markov LLS triple-graphs} and extend the definition of Markov triples
in Subsection~\ref{Generalised Markov triples}.
The properties of this generalisation are collected in the diagram of
Subsection~\ref{GeneralisedSection}
(see Theorem~\ref{generaliseTheorem}).
In Subsection~\ref{Uniqueness conjecture for Markov triples}
we recall the Uniqueness conjecture for Markov triples.
Finally in Subsection~\ref{A counterexample to the generalised uniqueness conjecture}
we show counterexamples to the generalised Markov conjecture
(see Examples~\ref{4-11} and~\ref{4-11-2}).

\subsection{Classical Markov theory in one diagram}
\label{Classical Markov theory in one diagram}

On the diagram in Figure~\ref{diagram.4}
we show all the maps that arise in classical Markov theory
for the discrete Markov spectrum.
Further in Section~\ref{GeneralisedSection} we generalise this diagram to the cases of
triple-graphs with different LLS sequences.

\vspace{2mm}

\begin{figure}
$$\epsfbox{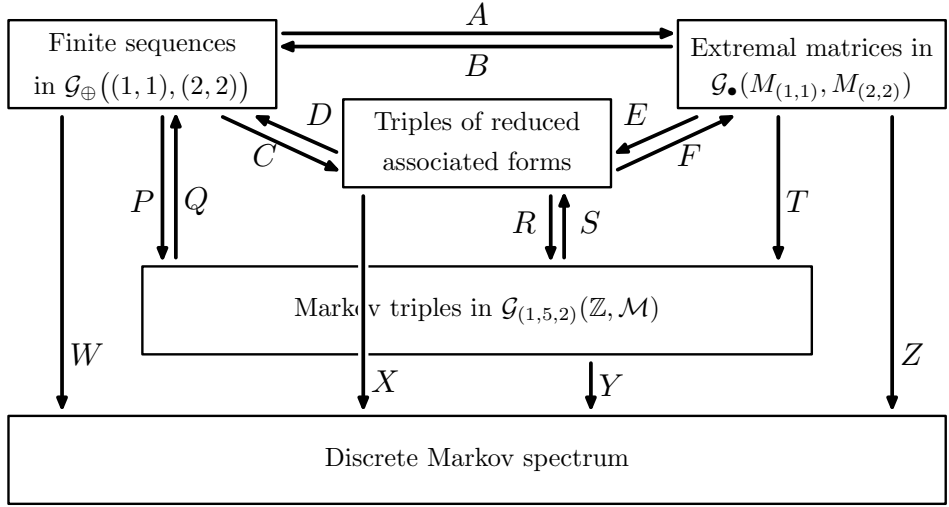}$$
		\caption{Classical Markov theory for extremal sequences.}
		\label{diagram.4}
\end{figure}

In fact Map~$C$ here is a definition of triples of associated forms:

\begin{definition}
Map $C$ of Figure~\ref{diagram.4} is induced by the  map which was considered in the periodic case above (see Figure~\ref{diagram.2}):
$$
(a_1,\ldots,a_n)
\mapsto
f_{a_1,\ldots,a_n}(x,y)=
K_1^{n-1}x^2+(K_1^n-K_2^{n-1})xy-K_2^{n}y^2.
$$
We say that the (ordered) triples of forms obtained by such map are {\it associated} to a triple of finite sequences.
\end{definition}

\begin{remark}\label{transforma}
These forms are simply related to the forms studied by A.~Markov.
After the following integer lattice preserving coordinate transformation:
$$
\left(
\begin{array}{c}
x\\
y
\end{array}
\right)
\mapsto
\left(
\begin{array}{c}
-x-2y\\
y
\end{array}
\right)
$$
the associated forms from the above definition are taken to the forms
of Theorem~\ref{MarkovSpectrum} by A.~Markov.
(In fact, this transformation corresponds to a one element shift in the LLS sequence.)
\end{remark}

\begin{remark}
Maps~$A$--$F$, $W$, $X$, $Z$ considered coordinatewise are actually restrictions of the corresponding maps in the diagram of Figure~\ref{diagram.2} in previous section. So we skip their description here.
\end{remark}

\begin{remark}\label{form-reconstruction2}
The central form in the associated triple forms uniquely determines the central LLS sequence (via the composition $B\circ F$). The other two forms
of the triple are uniquely reconstructed by the brute force
reconstruction of associated forms from reduced forms via building the tree of LLS sequences, as in Remark~\ref{Triple-reconstruction}.
\end{remark}

{\bf Map $P$:}
Consider a triple
$$
(\alpha,\beta,\gamma) \in \G_{\oplus}\big((1,1),(2,2)\big).
$$
then we set
$$
(\alpha,\beta,\gamma) \mapsto \big(\breve K(\alpha), \breve K(\beta), \breve K(\gamma)\big).
$$

\vspace{2mm}

{\bf Map $Q$:}
This map is extracted from the Theorem~\ref{MarkovSpectrum} by A.Markov.
Consider a triple $(a,M,b)$ where $M>b>a$.
Let $u$ be the least positive integer satisfying
either of the following two equations
\begin{equation}\label{EqThM}
\pm ua \equiv b \mod M.
\end{equation}
Let
$$
\frac{M}{u}=[a_1,\ldots,a_{2n-1}]
$$
be the odd regular continued fraction for $M/u$.
(Note that it is important
to take the odd continued fraction here.)
Then the period of the marked period for the corresponding LLS sequence is
$$
(a,M,b)
\mapsto
(a_{2n-1},\ldots,a_{1},2).
$$

\vspace{2mm}
{\bf Map $R$:}
$$
(f_1,f_2,f_3)
\mapsto
\Big(f_1(1,0),f_2(1,0), f_3(1,0)\Big).
$$
\vspace{2mm}

{\bf Map $S$:}
Consider a triple $(a,M,b)$ where we reorder the elements in the following way: $M>b>a$. As in Map~$Q$
let $u$ be defined by Equation~(\ref{EqThM}).
Set
$$
v=\frac{u^2+1}{M}.
$$
Then Map $S$ at the triple $(a,M,b)$ is defined as
$$
(a,M,b)
\mapsto
Mx^2+(M{+}2u)xy+(u{+}v{-}2M)y^2.
$$

\begin{remark}
Here the pair $(u,v)$ is defined as in Markov's theorem (Theorem~\ref{MarkovSpectrum}),
with an obvious change of coordinates inverse to the one introduced
in Remark~\ref{transforma}.
\end{remark}

\begin{remark}
Note that $u$ and $v$ are defined by the triple $(a,M,b)$.
The statement that $u$ and $v$ may be reconstructed entirely by $M$ is
equivalent to the uniqueness conjecture.
\end{remark}

{\bf Map $T$:}
Here we have a coordinate mapping between a triple of matrices
and a Markov triple
$$
\left(
\begin{array}{ll}
a&b\\
c&d\\
\end{array}
\right)
\mapsto
c.
$$

\vspace{2mm}

{\bf Map $Y$:}
This map is provided by Markov's theorem (Theorem~\ref{MarkovSpectrum}).
$$
(a,M,b)
\mapsto
\frac{\sqrt{9M^2-4}}{M}.
$$
The existence of the inverse to Map $Y$
is equivalent to the uniqueness conjecture (see Remark~\ref{Remark-unic}).
\vspace{2mm}

\begin{remark}
Since all the maps corresponding to arrows with opposite directions of the diagram in Figure~\ref{diagram.4} are inverse to each other,
the diagram is commutative. This gives a core for the classical Markov theory.
\end{remark}

\subsection{Markov LLS triple-graphs}
\label{Markov LLS triple-graphs}

Let us describe almost Markov and Markov LLS triple-graphs.

\vspace{2mm}

For real $a,b$ we set
$$
g_{a,b}=(x-ay)(x-by).
$$
Denote by $U_{+,+}^{a,b}$ the region defined by $x-ay>0$ and $x-by>0$.

\begin{definition}
Let $\mu$ and $\nu$ be two evenly-prime sequences of integers.
Let the following be true:

\begin{itemize}
\item $\langle\mu\rangle\prec(\nu)$ and $\langle\overline\mu\rangle\prec\langle\overline\nu\rangle$;

\item the global minima of $g_{[\langle\mu\rangle],
-[0;\langle\overline\mu\rangle]}$ at nonzero integer points
of the cone $U_{+,+}^{[\langle\mu\rangle],
-[0;\langle\overline\mu\rangle]}$ is attained at $(1,0)$.

\item the global minima of $g_{[\langle\nu\rangle],
-[0;\langle\overline\nu\rangle]}$ at nonzero integer points
of the cone $U_{+,+}^{[\langle\nu\rangle],-[0;\langle\overline\nu\rangle]}$ is attained at $(1,0)$.
\end{itemize}

Then the triple-graph $\G_\oplus(\mu,\nu)$
is called the {\it almost Markov LLS triple-graph}.
\end{definition}

\begin{definition}
We say that an even finite sequence $\alpha=(a_1,\ldots a_{2n})$ is {\it evenly palindromic}
if there exist an integer $k$ such that for every integer $m$ we have
$$
a_{k+m \mod 2n}=a_{k-m-1 \mod 2n}.
$$
\end{definition}

\begin{definition}
An almost Markov LLS triple-graph $\G_\oplus(\mu,\nu)$ is a {\it Markov LLS triple-graph}
if
\begin{itemize}
\item every sequence in every triple of $\G_\oplus(\mu,\nu)$ is evenly palindromic.
\end{itemize}
\end{definition}

\begin{example}
Let $m,n\ge 1$ and $p>q\ge 1$.
Then for the sequences
$$
\mu=\big((\underline{p})^{2n}\big), \qquad \nu=\big((\underline{q})^{2m}\big)
$$
the first three conditions are straightforward, and the palindrome condition is proved in~\cite{Matty1}.
Note that the case $p=1$, $q=2$, $m=n=1$ corresponds to the classical case of Markov tree $\G_\oplus\big((1,1),(2,2)\big)$.
\end{example}

\begin{example}
Let $m,n\ge 1$ and $p>q\ge 1$, and $\alpha$ be an evenly palindromic sequence.
Then
$$
\mu=\big((\underline{p})^{2n}\big), \qquad \nu=\big((\underline{p})^{2m}\alpha\big)
$$
satisfies the fourth condition.
Here one can construct many examples satisfying the first three conditions.
For instance, one can take
$$
\mu=(5,5,1,2,2,1) \qquad \hbox{and} \qquad \nu=(5,5).
$$
\end{example}

\begin{problem}
Describe all pairs of sequences $\mu$, $\nu$ for which
the triple-graph $\G_\oplus(\mu,\nu)$ is Markov/almost Markov.
\end{problem}

\begin{remark}
A very important property for a Markov triple-graph $\G_\oplus(\mu,\nu)$
is as follows.
For every element of every triple the corresponding reduced form  has
a normalised Markov minimum at $(1,0)$ (see Corollary~\ref{(1,0)-Minima-main}).
This property allow us to relate LLS sequences and corresponding matrices
directly with elements of the Markov spectrum.
\end{remark}

\begin{remark}
For an almost Markov graph $\G_\oplus(\mu,\nu)$ we have a slightly weaker statement:
for every element of every triple the corresponding reduced form has
a global maximum for all negative values of the form at integer points at $(1,0)$ (see Corollary~\ref{(1,0)-Minima}).
\end{remark}

\begin{remark}
Note that there is a nice expression for Map~$W$ 
in the case of $\mathcal{G}_{\oplus}((1,1),(2,2))$. Here for every sequence $(a_1,\ldots,a_n)$ in every triple of $\mathcal{G}_{\oplus}((1,1),(2,2))$ have
\[
W(a_1,\ldots,a_n)=a_n+[0;\langle a_1,\ldots,a_n\rangle]+[0;\langle a_{n-1},\ldots,a_1,a_n\rangle].
\]
\end{remark}

\subsection{Generalised Markov and almost Markov triples}
\label{Generalised Markov triples}

Starting from Markov/almost Markov LLS triple-graphs one can define generalised Markov/almost Markov graphs.
In this subsection we discuss the triple-graph structure for them.

\subsubsection{Generalised Markov/almost Markov graphs}

\begin{definition}
Consider a Markov/almost Markov LLS triple-graph $\G_\oplus(\mu,\nu)$
and replace each triple of sequences $(\alpha,\beta,\gamma)$
by a triple of integers  $\big(\breve K (\alpha),\breve K(\beta),\breve K(\gamma)\big)$.
The resulting triple-graph is called the {\it generalised Markov/almost Markov}
graph
and denoted by $T_{\mu,\nu}$.

The triples of  generalised Markov/almost Markov graph are called
{\it generalised Markov/almost Markov} triples.
\end{definition}

\begin{figure}
\epsfbox{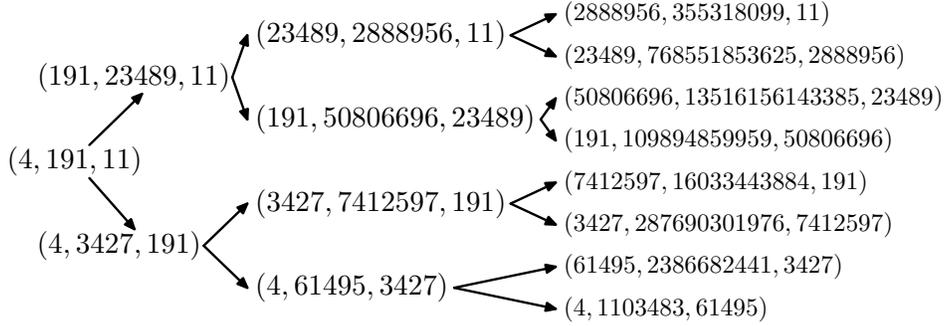}		
		\caption{The first 4 levels in the tree
                    $T_{(4,4),(11,11)}$.}
		\label{fig-mark-2}
\end{figure}

\begin{example}
In Figure~\ref{fig-mark-2}
we show the first four layers of the extended Markov triple-graph
$T_{(4,4),(11,11)}$. By construction we automatically have: $v_2>v_1,v_3$. Note that we do not require $v_2>v_3>v_1$ as it was for the Markov tree, so if a similar order is needed one should swap $v_1$ and $v_3$ if necessary.
\end{example}

\subsubsection{Triple-graph structure for generalised Markov and almost Markov graphs}

In order to define the triple-graph structure on $T_{\mu,\nu}$ we should extend
the operation
$$
\Sigma:(a,b,c) =3ac-b
$$
which is defined for the classical Markov tree $T_{(1,1),(2,2)}$.

\vspace{2mm}

The main difficulty here is that a straightforward operation on the graph $T_{\mu,\nu}$
is defined by sequences rather then by triples:
\begin{equation}\label{TildeSigma}
\tilde\Sigma\Big((\alpha,a),(\beta,b),(\gamma,c)\Big)=\breve K(\alpha\oplus\beta),
\end{equation}
where $(\alpha,\beta,\gamma)\in\G_\oplus(\mu,\nu)$ is the triple defining $(a,b,c)$.

\vspace{2mm}

Here one should remember all the sequences together with triples,
combining together $T_{\mu,\nu}$ and $\G_\oplus(\mu,\nu)$.
Hence we arrive to the following ternary
operation on the product $\z^\infty\times \z$:
$$
\otimes \Big((\alpha,a),(\beta,b),(\gamma,c)\Big)=
\Big(\alpha\oplus\beta, \tilde\Sigma\big((\alpha,a),(\beta,b),(\gamma,c)\big)
\Big).
$$
Now we can give the following definition.

\begin{definition}
Denote by $\G_\otimes (\mu,\nu)$ the triple-graph
$$
\G_{(\mu,\breve K(\mu)), (\mu\oplus\nu,\breve K(\mu\oplus\nu)), (\nu,\breve K(\nu))}
(\z^\infty\times \z,\otimes).
$$
\end{definition}

In fact as we show in Corollary~\ref{main cor},
we have
\begin{equation}\label{JustSigma}
\tilde\Sigma \Big((\alpha,a),(\beta,b),(\gamma,c)\Big)=
\frac{\breve K(\alpha\oplus\alpha)}{\breve K(\alpha)} b-c.
\end{equation}

\begin{remark}
In the classical case of $\G_\otimes \big((1,1),(2,2)\big)$ we have
$$
\frac{\breve K(\alpha\oplus\alpha)}{\breve K(\alpha)}=3a
$$
(see Proposition~\ref{useful-proposition} later)
and hence
$$
\tilde\Sigma \Big((\alpha,a),(\beta,b),(\gamma,c)\Big)=
\Sigma(a,b,c)=3ac-b.
$$
This establishes a straightforward equivalence between the triples of
$T_{(1,1),(2,2)}$ and the triples of $\G_\otimes \big((1,1),(2,2)\big)$.
\end{remark}

So there is a natural question here.

\begin{problem}\label{prob2}
For which triple-graphs $\G_\otimes (\mu,\nu)$ is the function
$
\tilde \Sigma
$
a function on triples $(a,b,c)$ (and not depend on $\alpha$)?
\end{problem}

This is equivalent to the following one.

\begin{problem}
For which pairs of even sequences $(\mu,\nu)$ does
every triple of $T_{\mu,\nu}$ occur only once in the tree.
\end{problem}

There is not much known here.
Some partial examples are given in Section~\ref{Recurrence relation for extended Markov trees},
see Example~\ref{SimpleRelationExample}.

\subsection{Generalised Markov theory in one diagram}\label{GeneralisedSection}

\begin{figure}
$$\epsfbox{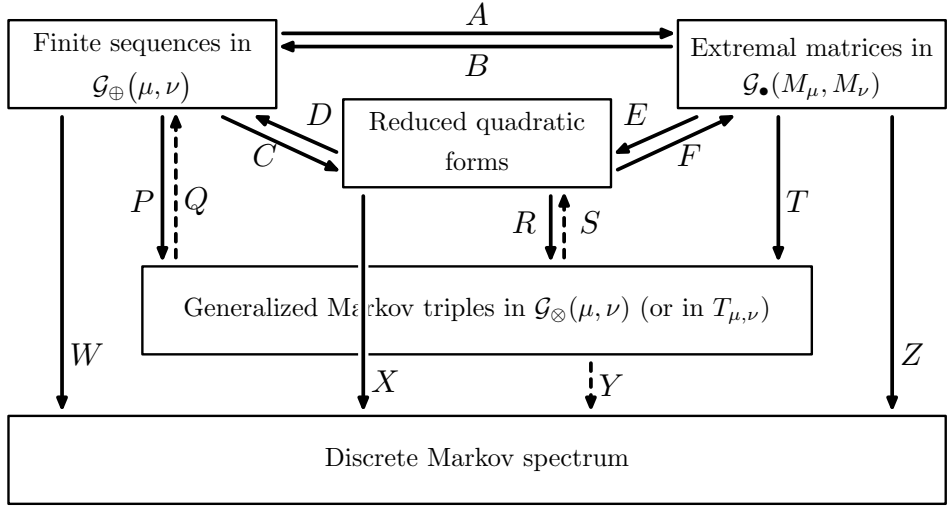}$$
		\caption{Extended Markov theory for extremal sequences.}
		\label{diagram.5}
\end{figure}

Finally for generalised Markov triple-graphs we have the following
diagram.

\vspace{2mm}

Let $\G_\otimes(\mu,\nu)$ be a triple-graph
for some generalised Markov triple-graph $T_{\mu,\nu}$.
The extension of classical Markov theory for this triple-graph
is presented in the diagram of Figure~\ref{diagram.5}.
As we will discuss later
in Subsection~\ref{Partial-answer},
some triple-graphs $\G_\otimes(\mu,\nu)$ can be reconstructed from
the triple-graph $T_{\mu,\nu}$.

\vspace{2mm}

In the diagram of Figure~\ref{diagram.5} most of the maps are literally extensions from the classical case of Figure~\ref{diagram.4}.
This follows directly from
Corollary~\ref{(1,0)-Minima-main} (which we show below).
All of the forms corresponding to generalised Markov triple-graphs
have a normalised Markov minimum at $(1,0)$, and therefore
by Definition~\ref{extremal} they are extremal.
This observation leads to the fact that the properties that A.~Markov has
detected on a particular example of discrete Markov spectrum
could be seen much wider.
As a consequence we have the following theorem.

\begin{theorem}\label{generaliseTheorem}
Let $G_\otimes(\mu,\nu)$ be a triple graph for some generalised
Markov triple-graph $T_{\mu,\nu}$
Then the maps of the diagram of Figure~\ref{diagram.5} are defined as follows.
\begin{itemize}
\item{
The maps
$$
A \hbox{--} F,W,X,Z
$$
are precisely the restrictions of the corresponding maps in the diagram of
Figure~\ref{diagram.2}.
}

\item{
The maps $P,R$ and $S$ are as in Figure~\ref{diagram.4}:
$$
\begin{array}{l}
P:(\alpha,\beta,\gamma) \mapsto \big(\breve K(\alpha), \breve K(\beta), \breve K(\gamma)\big);\\
R:
(f_1,f_2,f_3)
\mapsto
\Big(f_1(1,0),f_2(1,0), f_3(1,0)\Big);
\\
T:
\left\{\left(
\begin{array}{ll}
a_i&b_i\\
c_i&d_i\\
\end{array}
\right)
\Big| i=1,2,3
\right\}
\mapsto
(c_1,c_2,c_3).
\\
\end{array}
$$
}

\item
{
Finally: Map~$Q$ is trivial; while Maps $S$ and $Y$ are derived from maps
$C$ and $W$ respectively for the triple-graph $\G_\otimes(\mu,\nu)$.
}
\end{itemize}
\qed
\end{theorem}

This theorem confirms that a remarkable structure
discovered by Markov for the discrete Markov spectrum
has a general nature in geometry of numbers.

\vspace{2mm}

Let us finalise this section with several important remarks
concerning the diagram of Figure~\ref{diagram.5}.

\begin{remark}
Explicit expressions for Maps~$Q$, $S$, and $Y$
for $T_{\mu,\nu}$ (in case of existence) are
not known to the authors.
For this reason they are marked by dashed lines.
\end{remark}

\begin{remark}
Similar to the Markov settings, all the triples in the Generalised Markov diagram
are reconstructible by the middle element in a unique way (see Remark~\ref{Triple-reconstruction} and Remark~\ref{form-reconstruction2}).
\end{remark}

\begin{remark}\label{crutial-one}
Corollary~\ref{(1,0)-Minima-main} below is crucial for the maps $W$, $X$,
and $Z$ of this section.
It establishes a connection between sequences, forms, and matrices
on the one hand and the elements of Markov spectrum on the other hand.
\end{remark}

\begin{remark}
 Similar to the classical Markov case (see Figure~\ref{diagram.4}) the maps corresponding to arrows with opposite directions in the diagram in Figure~\ref{diagram.5} are inverse to each other.
 Therefore the diagram in Figure~\ref{diagram.5} is commutative.
\end{remark}

\subsection{Uniqueness conjecture for Markov triples}
\label{Uniqueness conjecture for Markov triples}

The following conjecture is due to G.~Frobenius.
\begin{conjecture}{\bf (Uniqueness conjecture, 1913.)}\label{Uniqueness-conjecture}
Every Markov number appears exactly once as the maximum in a Markov triple.
\end{conjecture}

\begin{remark}
The uniqueness conjecture is a slightly modified version of the fact that
$X$ is one-to-one with its image.

This in its turn is equivalent to the fact that
either $W$, or $X$, or $Z$ are one-to-one with their image.
\end{remark}

\begin{remark}\label{Remark-unic}
If we restrict $Y$ from Markov triples to single Markov numbers,
then we have one-to-one map with an image established by
$$
M
\mapsto
\frac{\sqrt{9M^2-4}}{M}
$$
(which is an increasing function of positive numbers).

So it is not known if it is possible reconstruct $M \mapsto (a,M,b)$.
As we show in Examples~\ref{4-11}
and~\ref{4-11-2}, this is not always the case for generalised Markov numbers.
\end{remark}

\subsection{A counterexample to the generalised uniqueness conjecture}
\label{A counterexample to the generalised uniqueness conjecture}

First of all we formulate the uniqueness conjecture for generalised
Markov numbers.

\begin{conjecture}{\bf (Generalised uniqueness conjecture, 2018.)}
\label{Uniqueness-conjecture2018}
Consider a generalised Markov triple-graph $T_{\mu,\nu}$.
Every generalised Markov number appears exactly once
as a maximal (middle) element in a Markov triple in $T_{\mu,\nu}$.
\end{conjecture}

Studying the extended Markov numbers of
$T_{(p,p),(q,q)}$
for different $(p,q)$ we have spotted two counterexamples
to the generalised uniqueness conjecture.
They are both for $p=4$ and $q=11$.
(Note that the first layers of the corresponding
triple-graph are shown in Figure~\ref{fig-mark-2}.)

\begin{example}\label{4-11}
First of all consider
$$
\alpha_1=\big(4,4,(\underline{11})^8\big)
\qquad
\alpha_2=\big((\underline{4})^{12},11,11\big)
$$
We have
$$
\breve K(\alpha_1)=\breve K(\alpha_2)=355318099,
$$
and hence Map~Y is not injective.
This implies that the associated forms
$$
\begin{array}{c}
355318099x^2+3856825285xy-928389367y^2\\
\hbox{and}\\
355318099x^2+3856242857xy-930136651y^2
\end{array}
$$
have the same normalised Markov minima.
Therefore the number $355318099$ appears at least twice as a maximal
(middle) element in the generalised Markov triples
of the triple-graph $\G_\oplus\big((4,4),(11,11)\big)$
\end{example}

\begin{example}\label{4-11-2}
The second example is for
$$
\alpha_1=\big(4,4,(\underline{11})^{14}\big)
\qquad
\alpha_2=\big((\underline{4})^{22},11,11\big).
$$
Here
$$
\breve K(\alpha_1)=\breve K(\alpha_2)=661068612553111.
$$
And therefore, for the same reason it is a counterexample
to the generalised uniqueness conjecture.

The corresponding forms are respectively
$$
\begin{array}{c}
661068612553111 x^2  + 7175615729089857 x y - 1727266560524267 y^2
\\
\hbox{and} \\
661068612553111 x^2  + 7174532122960713 x y - 1730517378911699 y^2.
\end{array}
$$
\end{example}

\begin{remark}
The arrangements of zero lines for the forms of Examples~\ref{4-11}
and~\ref{4-11-2} are not integer congruent to each other,
since they have different LLS sequences.
\end{remark}

\section{Related theorems and proofs}
\label{Related theorems and proofs}

In this section we prove several theorems that are used in generalised Markov theory.
In Subsection~\ref{Concatenation of sequences and global minima of the corresponding forms} we prove that any Markov LLS triple-graph $\G_\oplus(\mu,\nu)$ has Markov minima at $(1,0)$ and that these minima are unique up to the integer symmetries of the sail containing $(1,0)$.
Further in Subsection~\ref{Recurrence relation for extended Markov trees}
we prove that the triples of the triple-graph
$\G_\otimes(\mu,\nu)$ satisfy Equation~(\ref{JustSigma}) on page~\pageref{JustSigma}.
Finally in Subsection~\ref{Partial-answer} we prove the first results towards a solution
of Problem~\ref{prob2}.

\subsection{Concatenation of sequences and global minima of the corresponding forms}	
\label{Concatenation of sequences and global minima of the corresponding forms}

In this section we formulate and prove one of the central results of the current paper
(Corollary~\ref{(1,0)-Minima-main}),
which states that all the forms have a normalised Markov minimum (see Definition~\ref{Markov-minimum}).
This result is crucial for linking extended Markov numbers with normalised Markov minima for them.

\vspace{2mm}

Recall that for real $a,b$ we consider
$$
g_{a,b}=(x-ay)(x-by),
$$
and that $U_{+,+}^{a,b}$ denotes the region defined by $x-ay>0$ and $x-by>0$.

\begin{lemma}\label{lemma12345}
Let $a$ and $b$ be two real numbers.
Consider a point $v=(x,y)$ of the region $U_{+,+}^{a,b}$, and
let $\varepsilon$ be a positive real number.
Then the following hold.
\vspace{1mm}

{\it $($i$)$}
If $y> 0$ then
$$
\left\{
\begin{array}{l}
g_{a+\varepsilon,b}(v)< g_{a,b}(v)\\
g_{a,b+\varepsilon}(v)< g_{a,b}(v)\\
\end{array}
\right.
.
$$

\vspace{1mm}

{\it $($ii$)$}
If $y< 0$ then
$$
\left\{
\begin{array}{l}
g_{a+\varepsilon,b}(v)> g_{a,b}(v)\\
g_{a,b+\varepsilon}(v)> g_{a,b}(v)\\
\end{array}
\right.
.
$$

{\it $($iii$)$}
If $y= 0$ then
$$
\left\{
\begin{array}{l}
g_{a\pm\varepsilon,b}(v)= g_{a,b}(v)\\
g_{a,b\pm\varepsilon}(v)= g_{a,b}(v)\\
\end{array}
\right.
.
$$
\end{lemma}

\begin{proof}
Consider $v=(x,y)\in U_{+,+}^{a,b}$, and $\varepsilon >0$.
We have
$$
g_{a+\varepsilon,b}(x,y)-g_{a,b}(x,y)=-\varepsilon y(x-by).
$$
The sign of this expression is opposite to the sign of $y$
since $x{-}by>0$ for every point of the region $U_{+,+}^{a,b}$
and $\varepsilon>0$.
Therefore if $y > 0$ then the difference is negative,
and hence $g_{a+\varepsilon,b}(v)< g_{a,b}(v)$.

All the other three cases of~{\it$($i$)$} and~{\it$($ii$)$}
are analogous.

Item~{\it$($iii$)$} is trivial.
\end{proof}

\begin{remark}
Recall that we write $\big(a_1,\ldots, a_k, \langle b_1,\ldots, b_l\rangle \big)$ for an eventually
periodic sequence with a preperiod $(a_1,\ldots, a_k)$ and a period $(b_1,\ldots, b_l)$.
\end{remark}

\begin{proposition}\label{prop123}
Consider the following four infinite continued fraction expressions with positive integer elements:
$$
\begin{array}{l}
a_1=[0;p_1:p_2:\ldots ],\\
a_2=[0;q_1:q_2:\ldots ],\\
b_1=-[r_1:r_2:\ldots ],\\
b_2=-[s_1:s_2:\ldots ].\\
\end{array}
$$
Let $(p)_{n=1}^\infty  \prec  (q)_{n=1}^\infty$ and $(r)_{n=1}^\infty  \prec  (s)_{n=1}^\infty$.
Consider also
$$
v\in U_{+,+}^{a_1,b_1}\cap U_{+,+}^{a_2,b_2}.
$$
Then the following is true:
$$
\sign\big(g_{a_1,b_1}(v)-g_{a_2,b_2}(v)\big)=\sign(y).
$$
\end{proposition}

\begin{proof}
Since $(p)_{n=1}^\infty  \prec  (q)_{n=1}^\infty$, we have $a_1>a_2$.
Since $(r)_{n=1}^\infty  \prec  (s)_{n=1}^\infty$, we have $b_1>b_2$.
Now the statement of the proposition is a direct corollary of Lemma~\ref{lemma12345} applied twice.
\end{proof}

\begin{remark}
Note that here $a$ and $b$ are the cotangents of the slopes for  $g_{a,b}$
(while they are tangents for $f_{a,b}$ considered above).
\end{remark}

Let $\alpha$ be a finite sequence of numbers, denote by $\overline \alpha$ a sequence obtained from $\alpha$ by the inversion of the order of the elements.

\vspace{2mm}

For the remaining part of this subsection it is convenient to use the following definition.
\begin{definition}
Let $\alpha$ be a finite sequence of positive integers. A pair of numbers
$$
[\langle\alpha\rangle], -[0:\langle\overline{\alpha}\rangle]
$$
is called the {\it two-sided periodisation of $\alpha$} and denoted by
$[\langle\langle\alpha\rangle\rangle]$.
\end{definition}

For the proof of the next theorem we need the following direct corollary of
Proposition~\ref{prop123}. 

\begin{corollary}\label{coro1234}
Let $\mu$ and $\nu$ be two evenly-prime sequences of integers.
Assume that $\langle\mu\rangle\prec\langle\nu\rangle$ and $\langle\overline\mu\rangle\prec\langle\overline\nu\rangle$;
\vspace{1mm}

{
\noindent
$($a$)$ Consider
$$
v\in U_{+,+}^{[\langle\langle\mu\rangle\rangle]}\cap U_{+,+}^{[\langle\langle\mu\nu\rangle\rangle]}\quad (=U_{+,+}^{[\langle\langle\mu\rangle\rangle]}).
$$
Then
$$
\sign\big(g_{[\langle\langle\mu\nu\rangle\rangle]}(v)-g_{[\langle\langle\mu\rangle\rangle]}(v)\big)=\sign(y).
$$
}

{\noindent
$($b$)$ Consider
$$
v\in U_{+,+}^{[\langle\langle\nu\rangle\rangle]}\cap U_{+,+}^{[\langle\langle\mu\nu\rangle\rangle]}\quad (=U_{+,+}^{[\langle\langle\mu\nu\rangle\rangle]}).
$$
Then
$$
\sign\big(g_{[\langle\langle\mu\nu\rangle\rangle]}(v)-g_{[\langle\langle\nu\rangle\rangle]}(v)\big)=-\sign(y).
$$
}
\qed
\end{corollary}

\begin{remark}
Note that 
$$
U_{+,+}^{[\langle\langle\mu\rangle\rangle]}\cap U_{+,+}^{[\langle\langle\mu\nu\rangle\rangle]}=U_{+,+}^{[\langle\langle\mu\rangle\rangle]},
$$
since by Proposition~\ref{LastLabel169} the comparison $\langle\mu\rangle\prec\langle\mu\nu\rangle$ implies that 
$[\langle\mu\rangle]<[\langle\mu\nu\rangle]$.
By the same reason we have
$$
U_{+,+}^{[\langle\langle\nu\rangle\rangle]}\cap U_{+,+}^{[\langle\langle\mu\nu\rangle\rangle]}=U_{+,+}^{[\langle\langle\mu\nu\rangle\rangle]}.
$$
\end{remark}

We now continue with the following important result.

\begin{theorem}\label{(1,0)-Minima-one-sail}
Let $\mu$ and $\nu$ be two evenly-prime sequences of integers.
Assume that
\begin{itemize}

\item $\langle\mu\rangle\prec\langle\nu\rangle$ and $\langle\overline\mu\rangle\prec\langle\overline\nu\rangle$;

\item The global minima of $g_{[\langle\langle\mu\rangle\rangle]}$ at nonzero integer points
of the cone $U_{+,+}^{[\langle\langle\mu\rangle\rangle]}$ is attained at $(1,0)$.

\item The global minima of $g_{[\langle\langle\nu\rangle\rangle]}$ at nonzero integer points
of the cone $U_{+,+}^{[\langle\langle\nu\rangle\rangle]}$ is attained at $(1,0)$.

\end{itemize}

\vspace{1mm}

Then the global minima of $g_{[\langle\langle\mu\nu\rangle\rangle]}$ at nonzero integer points
of the cone $U_{+,+}^{[\langle\langle\mu\nu\rangle\rangle]}$ is attained at $(1,0)$.

This global minimum is the unique global minimum on the sail for $g_{[\langle\langle\mu\nu\rangle\rangle]}$ up to lattice preserving translations of the sail.
\end{theorem}

\begin{proof}
From the general theory of periodic continued fractions it follows that
the infimum of $g_{[\langle\langle\mu\nu\rangle\rangle]}$
for the integer points of the cone $U_{+,+}^{[\langle\langle\mu\nu\rangle\rangle]}$
except the origin is
attained at some vertex of the sail for this cone.

\vspace{1mm}

Then every fundamental domain of the sail with respect to the action of
the positive Dirichlet group of $\SL(2,\z)$-operators (which preserves the cone) contains at least one vertex which is a global minimum.
In particular there is a global minimum at the period
which is the union of the sails for
$[\mu]$ and $-[0;\overline\nu]$.
Note that both of these sail parts contain $(1,0)$ as a vertex.

Let $(x,y)\ne(1,0)$ be a vertex of the part of the sail for $[0,\mu]$.
First, note that
$$
(x,y)\in U_{+,+}^{[\langle\langle\mu\rangle\rangle]}\cap U_{+,+}^{[\langle\langle\mu\nu\rangle\rangle]}=
U_{+,+}^{[\langle\langle\mu\rangle\rangle]}.
$$
Secondly, $y>0$, since the part of the sail for $[0,\mu]$ except $(1,0)$
is contained in the halfplane $y>0$.

Then we have
$$
g_{[\langle\langle\mu\nu\rangle\rangle]}(1,0)=
g_{[\langle\langle\mu\rangle\rangle]}(1,0)<
g_{[\langle\langle\mu\rangle\rangle]}(x,y)<
g_{[\langle\langle\mu\nu\rangle\rangle]}(x,y).
$$
The first equality and the third inequality hold by Proposition~\ref{coro1234} (since $y>0$).
The second inequality holds by the assumption of the theorem.

Similarly for the case of $(x,y)\ne(1,0)$ being a vertex in the part of the sail for $[\overline{\beta}]$ we have
$$
g_{[\langle\langle\mu\nu\rangle\rangle]}(1,0)=
g_{[\langle\langle\nu\rangle\rangle]}(1,0)<
g_{[\langle\langle\nu\rangle\rangle]}(x,y)<
g_{[\langle\langle\mu\nu\rangle\rangle]}(x,y)
$$
(here $y<0$).

Therefore $(1,0)$ is a global minimum of the sail period
which is the union of sails for $-[0;\overline{\nu}]$
and $[\mu]$.

\vspace{2mm}

Since all inequalities are strict, we have the uniqueness of global minimum at each period of the sail.
\end{proof}

\begin{corollary}\label{(1,0)-Minima}
Let $\mu$ and $\nu$ be two evenly-prime sequences of integers.
Assume that

\begin{itemize}
\item $\langle\mu\rangle\prec\langle\nu\rangle$ and $\langle \overline\mu\rangle\prec\langle\overline\nu\rangle$;

\item $\langle\mu\nu\rangle$ is evenly palindromic;

\item the global minima of $g_{[\langle\langle\mu\rangle\rangle]}$ at nonzero integer points
of the cone $U_{[\langle\langle\mu\rangle\rangle]}$ is attained at $(1,0)$.

\item the global minima of $g_{[\langle\langle\nu\rangle\rangle]}$ at nonzero integer points
of the cone $U_{+,+}^{[\langle\langle\nu\rangle\rangle]}$ is attained at $(1,0)$.
\end{itemize}

\vspace{1mm}

Then the global minima of $g_{[\langle\langle\mu\nu\rangle\rangle]}$ at nonzero integer points is attained at $(1,0)$.

This global minimum is the unique global minimum on the sail for $g_{[\langle\langle\mu\nu\rangle\rangle]}$ up to lattice preserving translations of the sail.
\end{corollary}


\begin{proof}
From the general theory of periodic continued fractions it follows that
the infimum of $g_{[\langle\langle\mu\nu\rangle\rangle]}$ for the points $\z^2\setminus (0,0)$ is
attained at some vertex of the sail of one of the cones in the complement to
$g_{[\langle\langle\mu\nu\rangle\rangle]}=0$.

Since the LLS of $g_{[\langle\langle\mu\nu\rangle\rangle]}$ is evenly palindromic,
the sails of dual cones for $g_{[\langle\langle\mu\nu\rangle\rangle]}$ are integer congruent (see Definition~\ref{IntCong}).
Hence by~\cite[Theorem 4.10]{oleg1} the cones are integer congruent,
and therefore every cone for $g_{[\langle\langle\mu\nu\rangle\rangle]}$ contains global minima
of $|g_{[\langle\langle\mu\nu\rangle\rangle]}|$ (i.e. a Markov minimum).

\vspace{2mm}

Now both statements of Corollary~\ref{(1,0)-Minima}
follows directly from Theorem~\ref{(1,0)-Minima-one-sail}.

\end{proof}

\begin{example}
It is rather simple to see that the conditions of the corollary are very important.
For instance,
$$
\mu=(1,1,2,2,2,2) \qquad \hbox{and} \qquad \nu=(1,1,2,2)
$$
fulfils all the conditions except the fact that
$\langle\overline \mu\rangle\succ \langle\overline\nu\rangle$.
Then the form corresponding to $\mu\oplus\nu$ will not have a global minimum at $(1,0)$.
Indeed, set
$$
\xi=\mu\oplus\nu=(1,1,2,2,2,2,1,1,2,2),
$$
then
$$
g_{[\langle\langle\xi\rangle\rangle]}(x,y)=x^2+\frac{787}{437}xy-\frac{611}{437}y^2.
$$
We have
$$
g_{[\langle\langle\xi\rangle\rangle]}(1,0)=1, \quad \hbox{and} \quad
g_{[\langle\langle\xi\rangle\rangle]}(17,29)=\frac{433}{437}<1.
$$
\end{example}

\begin{corollary}\label{(1,0)-Minima-main}
Let $\G_\oplus(\mu,\nu)$ be a Markov LLS triple-graph.
Then the following are true

\begin{itemize}
\item For every element of every triple of $\G_\oplus(\mu,\nu)$ the
corresponding form attains a Markov minimum
precisely at $(1,0)$.

\item This Markov minimum is the unique Markov minimum on the sail for the form up to lattice preserving translations of the sail.
\end{itemize}
\end{corollary}

\begin{proof}
Recall that if $\G_\oplus(\mu,\nu)$
is a Markov LLS triple-graph then $\mu$ and $\nu$ are evenly-prime sequences of integers that satisfy
\begin{itemize}

\item $\langle\mu\rangle\prec(\nu)$ and $\langle\overline\mu\rangle\prec\langle\overline\nu\rangle$;

\item the global minima of $g_{[\langle\langle\mu\rangle\rangle]}$ at nonzero integer points
of the cone $U_{+,+}^{[\langle\langle\mu\rangle\rangle]}$ is attained at $(1,0)$;

\item the global minima of $g_{[\langle\langle\nu\rangle\rangle]}$ at nonzero integer points
of the cone $U_{+,+}^{[\langle\langle\nu\rangle\rangle]}$ is attained at $(1,0)$;

\item all the periodic LLS sequences associated to sequences of $\G_\oplus(\mu,\nu)$ are evenly palindromic.
\end{itemize}
\vspace{1mm}

Now we inductively apply Corollary~\ref{(1,0)-Minima} to all triples of $\G_\oplus(\mu,\nu)$. This concludes the proof.
\end{proof}

\begin{remark}
Here we should note that the first three conditions for Markov LLS triple-graphs are on the first elements
$\mu$ and $\nu$, while the palindrome condition is for the whole triple-graph $\G(\mu,\nu)$.
The last is usually rather hard to check.
\end{remark}

\begin{remark}
Note that the property of Corollary~\ref{(1,0)-Minima-main} is of main importance for the diagram of Figure~\ref{diagram.5} as it clarifies the maps $W$, $X$,
and $Z$ (see Remark~\ref{crutial-one}).
\end{remark}

\begin{remark}
For almost Markov triple-graphs a similar statement holds for the sails containing the point $(1,0)$.
\end{remark}

\subsection{Recurrence relation for extended Markov trees}
\label{Recurrence relation for extended Markov trees}

In this section we prove the Expression~(\ref{JustSigma})
for the operation $\tilde \Sigma$ in the definition of
$\G_\otimes(\mu,\nu)$.

We start with the following
general result on integer sines of periods.
Recall that we denote the concatenation of sequences
$\alpha$ and $\beta$ by $\alpha\beta$, and that
we use the following notation
$$
		\breve K(d_1,\ldots,d_s)=K(d_1,\ldots,d_{s-1}).
$$

\begin{theorem}\label{cab prop}
Let $\alpha$, $\beta$, and $\gamma$ be the following sequences of positive integers
$$
                \begin{aligned}
                \alpha&=(a_1,\ldots,a_{2n}),\\
                \lambda&=(b_1,\ldots,b_{l}),\\
                \rho&=(c_1,\ldots,c_{r}),
                \end{aligned}
$$
with $n$, $m$, and $r\in\Z_+$. Then we have that
\begin{equation}\label{cab eq}
\frac{\breve K(\alpha^2)}{\breve K(\alpha)}=\frac{\breve K(\lambda\alpha^2\rho)+
\breve K(\lambda\rho)}{\breve K(\lambda\alpha\rho)}.
\end{equation}
\end{theorem}

	For the proof we need the following two lemmas.

\begin{lemma}{\bf(R.L.~Graham, et al.~\cite{conc}.)} \label{cont rel}
		Let $(a_1,\ldots,a_n)$ be an arbitrary sequence
and $1\leq k\leq n$. Then we have
$$
K_1^n(\alpha)=K_1^k(\alpha)K_{k+1}^n(\alpha)+K_1^{k-1}(\alpha)K_{k+2}^n(\alpha).
$$
\end{lemma}
	
\begin{proof}
Let
$$
               \begin{aligned}
                \alpha&=(a_1,\ldots,a_{k}),\\
                \beta&=(a_{k+1},\ldots,a_{n}).
                \end{aligned}
$$
By Proposition~\ref{BasicPropertiesReduced}$($ii$)$ we have
\begin{equation}\label{zzz678}
M_{\alpha \beta}=M_\alpha \cdot M_\beta
\end{equation}
Now by the explicit definition of $M_\alpha$, $M_\beta$, and $M_{\alpha \beta}$, the statement of lemma is the equation
for the lower right element in Equation~(\ref{zzz678}).
\end{proof}

Now we give the proof of Theorem \ref{cab prop}.

\begin{lemma}\label{int rec rel}
For a sequence of positive integers
$\alpha=(a_1,\ldots,a_n)$
the value
$$
\frac{\breve K(\alpha^2)}{\breve K (\alpha)}=K_1^{n}(\alpha)+K_2^{n-1}(\alpha).
$$
(In particular the ratio is a positive integer.)
\end{lemma}
	
\begin{proof}
By Lemma~\ref{cont rel} applied to
$\alpha=\beta=(a_1,\ldots,a_n)$ we have
$$
\begin{aligned}
\frac{\breve K(\alpha^2)}{\breve K (\alpha)}
&=
\frac{K_1^{n}(\alpha^2)K_{n+1}^{2n-1}(\alpha^2)+K_2^{n-1}(\alpha^2)
K_{n+2}^{2n-1}(\alpha^2)  }{K_2^{n-1}(\alpha)}
\\
&=
\frac{K_1^{n}(\alpha)K_{1}^{n-1}(\alpha)+K_2^{n-1}(\alpha)K_{2}^{n-1}(\alpha)  }{K_2^{n-1}(\alpha)}
=
K_1^{n}(\alpha)+K_2^{n-1}(\alpha).
\end{aligned}
$$		
\end{proof}

\begin{remark}
Note that the definitions we have the following interesting property:
$$
\trace(M_\alpha)=K_1^n(\alpha)+K_2^{n-1}(\alpha).
$$
It plays an important role in the theory of Cohn matrices for the classical Markov tree (e.g. see in~\cite{Aigner1989}).
\end{remark}

	\begin{proof}[Proof of Theorem \ref{cab prop}]
From Lemma~\ref{int rec rel} Equation~(\ref{cab eq}) is equivalent to
$$
K(\alpha)+K_2^{n-1}(\alpha)=\frac{\breve K(\lambda\alpha^2\rho)+\breve K(\lambda\rho)}{\breve K(\lambda\alpha\rho)}.
$$
Therefore, we have
		
\begin{equation} \label{cab int. eq 1}
\breve K(\lambda\alpha\rho) K(\alpha)+\breve K(\lambda\alpha\rho)K_2^{n-1}(\alpha)-
\breve K(\lambda\alpha^2\rho)-\breve K(\lambda\rho)=0.
\end{equation}
		We use the following relations in our proof, coming from
Lemma~\ref{cont rel}:
\begin{equation} \label{cab eq 1}
		\breve K(\lambda\alpha^2\rho)=
            K(\lambda\alpha)\breve K(\alpha\rho)+\breve K(\lambda\alpha)
              K_2^{n+r-1}(\alpha\rho);
\end{equation}
\begin{equation} \label{cab eq 2}
		\breve K(\lambda\alpha\rho)=K(\lambda)\breve K(\alpha\rho)+
                   \breve K(\lambda)K_2^{n+r-1}(\alpha\rho);
\end{equation}
\begin{equation} \label{cab eq 3}
		\breve K(\lambda\alpha)=
          K(\lambda)\breve K(\alpha)+\breve K(\lambda)K_2^{n-1}(\alpha);
\end{equation}
\begin{equation} \label{cab eq 4}
		K(\lambda\alpha)=
          K(\lambda)K(\alpha)+\breve K(\lambda)K_2^{n}(\alpha);
\end{equation}
Substituting Equations~(\ref{cab eq 1}), (\ref{cab eq 2}),  and~(\ref{cab eq 3}), to Equation (\ref{cab int. eq 1}) we get
$$
\begin{array}{l}
\breve K(\lambda\alpha\rho) K(\alpha)+
\Big(K(\lambda)\breve K(\alpha\rho)+
                   \breve K(\lambda)K_2^{n+r-1}(\alpha\rho)\Big)K_2^{n-1}(\alpha)-\\
\Big(
 K(\lambda\alpha)\breve K(\alpha\rho)+
 \Big(K(\lambda)\breve K(\alpha)+\breve K(\lambda)K_2^{n-1}(\alpha)\Big)K_2^{n+r-1}(\alpha\rho)
 \Big)
\\
-\breve K(\lambda\rho)=0.
\end{array}
$$
After cancelling reciprocal terms we have
\begin{equation} \label{cab int. eq 2}
\begin{array}{l}
\breve K(\lambda\alpha\rho) K(\alpha)+
K(\lambda)\breve K(\alpha\rho)K_2^{n-1}(\alpha)-
\\ K(\lambda\alpha)\breve K(\alpha\rho)-
 K(\lambda)\breve K(\alpha)K_2^{n+r-1}(\alpha\rho)
-\breve K(\lambda\rho)=0.
\end{array}
\end{equation}
		Now substituting Equations~(\ref{cab eq 2}) and~(\ref{cab eq 4})
to Equation~(\ref{cab int. eq 2}) we obtain
\begin{equation} \label{cab int. eq 3}
\begin{array}{l}
\Big(K(\lambda)\breve K(\alpha\rho){+}
                   \breve K(\lambda)K_2^{n+r-1}(\alpha\rho)\Big)
K(\alpha)+
K(\lambda)\breve K(\alpha\rho)K_2^{n-1}(\alpha)-
\\
\Big(K(\lambda)K(\alpha)+\breve K(\lambda)K_2^{n}(\alpha)\Big)
\breve K(\alpha\rho)-
 K(\lambda)\breve K(\alpha)K_2^{n+r-1}(\alpha\rho)
-
\\
\breve K(\lambda\rho)=0.
\end{array}
\end{equation}
After cancelling reciprocal terms we have
$$
\begin{array}{l}
\breve K(\lambda)K_2^{n+r-1}(\alpha\rho)
K(\alpha)+
K(\lambda)\breve K(\alpha\rho)K_2^{n-1}(\alpha)-
\\
\breve K(\lambda)K_2^{n}(\alpha)
\breve K(\alpha\rho)-
 K(\lambda)\breve K(\alpha)K_2^{n+r-1}(\alpha\rho)
-
\breve K(\lambda\rho)=0.
\end{array}
$$
The last equation is equivalent to
\begin{equation}\label{cab int. eq 4}
\begin{array}{l}
K(\lambda)
\Big(
\breve K(\alpha\rho)K_2^{n-1}(\alpha)
-\breve K(\alpha)K_2^{n+r-1}(\alpha\rho)
\Big)+
\\
\breve K(\lambda)
\Big(
K_2^{n+r-1}(\alpha\rho)
K(\alpha)-K_2^{n}(\alpha)
\breve K(\alpha\rho)
\Big)-
\\
\breve K(\lambda\rho)=0.

\end{array}
\end{equation}
Let us study the expressions in the brackets of the last equation.
First of all, we have
$$
\begin{array}{l}
\breve K(\alpha\rho)K_2^{n-1}(\alpha)-\breve K(\alpha)K_2^{n+r-1}(\alpha\rho)=
\\
\Big(K(\alpha)\breve K(\rho)+\breve K(\alpha)K_2^{r-1}(\rho)\Big)K_2^{n-1}(\alpha)-
\\
\qquad
\breve K(\alpha)\Big(
 K_2^n(\alpha)\breve K(\rho)+K_2^{n-1}(\alpha)K_2^{r-1}(\rho)
\Big)
=
\\
\Big(
K(\alpha)K_2^{n-1}(\alpha)
-
\breve K(\alpha)
 K_2^n(\alpha)
\Big)
\breve K(\rho)=
\\
\breve K(\rho).
\end{array}
$$
Secondly, it holds that
$$
\begin{array}{l}
K_2^{n+r-1}(\alpha\rho)
K(\alpha)-K_2^{n}(\alpha)
\breve K(\alpha\rho)=\\
\Big(
 K_2^n(\alpha)\breve K(\rho)+K_2^{n-1}(\alpha)K_2^{r-1}(\rho)
\Big)
K(\alpha)-
\\
\qquad
K_2^{n}(\alpha)
\Big(K(\alpha)\breve K(\rho)+\breve K(\alpha)K_2^{r-1}(\rho)\Big)=\\
		(K(\alpha)K_2^{n-1}(\alpha)-K_2^n(\alpha)\breve K(\alpha))K_2^{r-1}(\rho)=\\
		=K_2^{r-1}(\rho),
\end{array}
$$
		Using this, Equation (\ref{cab int. eq 4}) becomes
$$
		K(\lambda)\breve K(\rho)+\breve K(\lambda)K_2^{r-1}(\rho)-\breve K(\lambda\rho)=0,
$$
		which is true, by Lemma~\ref{cont rel}.
	\end{proof}

Let us finally reformulate the above theorem
in terms of the map $\tilde\Sigma$.	

\begin{corollary} \label{main cor}
Let $a=\breve K(\alpha)$, $b=\breve K(\alpha\beta)$ and $c=\breve K (\beta)$ then
$$
\tilde \Sigma\Big((a,\alpha),(b,\alpha\beta),(c,\beta)\Big)=
\frac{\breve K(\alpha^2)}{\breve K(\alpha)} b-c.
$$
\end{corollary}

\begin{proof}
We have
$$
\begin{aligned}
\tilde \Sigma\Big((a,\alpha),(b,\alpha\beta),(c,\beta)\Big)
&=
\breve K(\alpha^2\beta)=
\frac{\breve K(\alpha^2)}{\breve K(\alpha)}\breve K(\alpha\beta)-\breve K(\beta)\\
&=
\frac{\breve K(\alpha^2)}{\breve K(\alpha)} b-c.\\
\end{aligned}
$$
Here the first equality follows directly from the definition (see Expression~(\ref{TildeSigma})); the second equality holds
by Theorem~\ref{cab prop}.
\end{proof}

\begin{remark}
Let us mention once more that
the value $\breve K(\alpha^2)/\breve K(\alpha)$ is always a positive integer,
see Lemma \ref{int rec rel}.
\end{remark}

\subsection{Partial answer to Problem~\ref{prob2}}
\label{Partial-answer}

Let us finally prove the following interesting statement that gives a few partial answers to Problem~\ref{prob2}.
(Recall that a sequence $\alpha$ is palindromic if $\alpha=\overline \alpha$.)

	\begin{proposition} \label{useful-proposition}
		For a sequence
		\[
		\alpha=(1,p,a_3,\ldots,a_{n-2},p+1,q),
		\]
		with the subsequence $(a_3,\ldots,a_{n-2})$ being palindromic, we have
		\[
		\frac{\breve K(\alpha^2)}{\breve K (\alpha)}=(q+1){\breve K (\alpha)}.
		\]
	\end{proposition}
	
	\begin{proof}
		We start with
		$$
		\begin{aligned}
		 K_1^{n-2}(\alpha)&=K(1,p,a_3,\ldots,a_{n-2})=K(a_{n-2},\ldots,a_3,p,1)\\
		&=K(a_{3},\ldots,a_{n-2},p+1)=K_3^{n-1}.
		\end{aligned}
		$$
Here the second equality is a general property of continuant. In the third equality we us the fact that $(a_3,\ldots,a_{n-2})$ is palindromic; and
 a general classical fact on continuants that
$$
K(b_1\ldots,b_m,1)=
K(b_1\ldots,b_m{+}1).
$$

\vspace{2mm}

Finally we have
		$$
		\begin{aligned}
		\frac{\breve K(\alpha^2)}{\breve K(\alpha)}
        &=K_1^n(\alpha)+K_2^{n-1}(\alpha)\\
		&=\Big(qK_1^{n-1}(\alpha)+K_1^{n-2}\Big)+K_2^{n-1}\\
        &=qK_1^{n-1}(\alpha)+\Big(K_3^{n-1}+1\cdot K_2^{n-1}\Big)\\
		&=q K_1^{n-1}(\alpha)+K_1^{n-1}\\
        &=(q+1)\breve K(\alpha).
		\end{aligned}
        $$
		The first equality is true by Lemma~\ref{int rec rel},
the second and the fourth are the general properties of continuants,
and the third is by the identity shown above.
\end{proof}

\begin{example}\label{SimpleRelationExample}
For the case of trees $\G_\otimes \big((\underline{1})^{2n},(\underline{2})^{2m}\big)$, where $m,n\ge 1$,
we have precisely all the condition of Proposition~\ref{useful-proposition}
(see palindromicity in~\cite{cus1} and~\cite{Matty1}).
Hence the operation can be written in terms of triples:
$$
\tilde \Sigma\Big((a,\alpha),(b,\beta),(c,\gamma)\Big)=3ac-b.
$$
This gives a particular answer to Problem~\ref{prob2} above.
In particular in the case $n=m=1$ we have the relation for the classical case of
Markov trees.
In all these cases the data of Markov triples $T_{\mu,\nu}$
coincides with the data of the triple-graphs $\G_\otimes(\mu,\nu)$.
\end{example}

Let us formulate the  following question (in fact it is a part of Problem~\ref{prob2} above).

\begin{problem}
For which pairs of even sequences $(\mu,\nu)$ can we apply Proposition~\ref{useful-proposition} for all triples of $\G_\oplus(\mu,\nu)$.
\end{problem}

\vspace{2mm}

{\noindent
{\bf Acknowledgements.}
The first author is partially supported by EPSRC grant EP/N014499/1 (LCMH).
We are grateful to A.~Gorodentsev for introducing us to classical theory of Markov spectrum.
}

	\bibliographystyle{plain}
	\bibliography{biblio,biblio2}

\vspace{2cm}

\end{document}